\documentclass[11pt,reqno]{article}
\usepackage[activeacute,english]{babel}

\usepackage[applemac]{inputenc}

\usepackage{amsfonts}
\usepackage{amsmath,amsthm,bm}
\usepackage{amssymb}
\usepackage{graphics}
\usepackage{tikz}
\usepackage{float}
\usepackage{xcolor}

\usepackage{hyperref}
\usepackage{cleveref}

\DeclareRobustCommand{\SkipTocEntry}[4]{}

\usepackage[margin=2.5cm]{geometry}

\newtheorem{theorem}{Theorem}[section]

\newtheorem{proposition}[theorem]{Proposition}
\newtheorem{lemma}[theorem]{Lemma}
\newtheorem{corollary}[theorem]{Corollary}

\theoremstyle{definition}

\theoremstyle{remark}
\newtheorem{remark}[theorem]{Remark}

\renewcommand{\L}{{\mathcal{L}}}
\newcommand{\G}{{\mathcal{G}}}

\newcommand{\R}{{\mathbb R}}

\newcommand{\D}{{\Delta}}
\newcommand{\T}{{\mathbb T}}
\newcommand{\V}{\Vert}
\newcommand{\Sph}{{\mathbb S}^2}
\renewcommand{\O}{{\Omega}}
\renewcommand{\o}{{\omega}}
\newcommand{\U}{\gamma}
\renewcommand{\d}{\mathrm{d}}
\newcommand{\e}{\varepsilon}
\renewcommand{\u}{\textbf{u}}
\renewcommand{\a}{\alpha}
\renewcommand{\b}{\beta}
\renewcommand{\P}{\Psi_*}

\allowdisplaybreaks[1]

\begin{document}
\title{On zonal steady solutions to the \\
2D Euler equations on the rotating unit sphere}
\author{Marc Nualart\footnote{\text{Department of Mathematics, Imperial College London, London, SW7 2AZ, UK. m.nualart-batalla20@imperial.ac.uk}}} 

\maketitle
\begin{abstract}
The present paper studies the structure of the set of stationary solutions to the incompressible Euler equations on the rotating unit sphere that are near two basic zonal flows: the zonal Rossby-Haurwitz solution of degree 2 and the zonal rigid rotation $Y_1^0$ along the polar axis. 

We construct a new family of non-zonal steady solutions arbitrarily close in analytic regularity to the second degree zonal Rossby-Haurwitz stream function, for any given rotation of the sphere. This shows that any non-linear inviscid damping to a zonal flow cannot be expected for solutions near this Rossby-Haurwitz solution.

On the other hand, we prove that, under suitable conditions on the rotation of the sphere, any stationary solution close enough to the rigid rotation zonal flow $Y_1^0$ must itself be zonal, witnessing some sort of rigidity inherited from the equation, the geometry of the sphere and the base flow. Nevertheless, when the conditions on the rotation of the sphere fail, the set of solutions is much richer and we are able to prove the existence of both explicit stationary and travelling wave non-zonal solutions bifurcating from $Y_1^0$, in the same spirit as those emanating from the zonal Rossby-Haurwitz solution of degree 2.
\end{abstract}
\section{Introduction}
In this paper we consider the incompressible Euler equation on the unit sphere $\Sph$ rotating around the polar axis with angular velocity $\widetilde{\U}\in\R$. In vorticity formulation it reads
\begin{equation}\label{Euler Sphere}
\partial_t\O+U\cdot\nabla(\O-2\widetilde{\U}\cos\theta)=0, \quad U=\nabla^\perp \Psi, \quad \D\Psi = \O,
\end{equation}
where $U$ denotes the divergence-free velocity field of the fluid tangent to the unit sphere $\Sph$, and $\O$ and $\Psi$ are its associated vorticity and stream-function, respectively. The term $-U\cdot \nabla(2\widetilde{\U}\cos\theta)$ accounts for the Coriolis force due to the rotation of the sphere. Here, we parametrize the unit sphere $\Sph$ by the usual spherical coordinate system
\begin{equation*}
x=\sin\theta\cos\varphi,\quad y=\sin\theta\sin\varphi, \quad z=\cos\theta,
\end{equation*}
where $\theta\in[0,\pi]$ and $\varphi\in[0,2\pi)$ are the colatitude and longitude, respectively. See Section \ref{Preliminaries} below for a more precise definition of the differential operators terms in \eqref{Euler Sphere} and the choice of the coordinate chart. The stream functions of stationary solutions to \eqref{Euler Sphere} satisfy
\begin{equation}\label{Steady Psi Sphere}
\nabla^\perp\Psi\cdot\nabla(\Delta \Psi-2\widetilde{\U}\cos\theta)=0.
\end{equation}
In this direction, any solution $\Psi$ to  
\begin{equation}\label{Steady F Sphere}
\D\Psi-2\widetilde{\U}\cos\theta=F(\Psi),
\end{equation} 
for some $F\in C^1$ automatically satisfies \eqref{Steady Psi Sphere} and is thus a stationary solution to the Euler equation.

The eigenfunctions of the Laplace-Beltrami operator $\D$ in $\Sph$ are the spherical harmonics and are such that $\D\Psi = -\lambda_n\Psi$, with $\lambda_n=n(n+1)$ being the corresponding eigenvalues (see Section \ref{Preliminaries}). In particular, we denote by $Y_n^0=Y_n^0(\theta)$ the spherical harmonic that is a zonal function (i.e., a function that only depends on the colatitude $\theta$) and satisfies $\D Y_n^0 = -\lambda_n Y_n^0$. 

Moreover, we introduce the zonal Rossby-Haurwitz stream function of degree $n$, which is given by 
\begin{equation*}
\Psi_n= \b Y_n^0+\frac{2\widetilde{\U}}{n(n+1)-2}\cos\theta, \quad \b\neq 0.
\end{equation*}
Together with the choice $F(\Psi)=-n(n+1)\Psi$, it solves \eqref{Steady F Sphere} and therefore it is an explicit steady solution to the Euler equation. The Rossby-Haurwitz flows on the rotating sphere are regarded as the direct analogues of the Kolmogorov flows on the torus, since they are, modulo rotation corrections, generated by the eigenfunctions of the Laplace-Beltrami operator, or spherical laplacian. Moreover, the flows produced by the Rossby-Haurwitz stream-functions are of interest from the meteorological viewpoint, since they are usually present in the atmospheres of the earth and the outer planets of the solar system, see \cite{Dowling, He}. Finally, note that zonal stream functions, of the form $\Psi = \Psi(\theta)$, solve \eqref{Steady Psi Sphere} for all $\U\in\R$, their associated velocity being $U=-\partial_\theta \Psi(\theta)\textbf{e}_\varphi$. As such, $Y_n^0$ solves \eqref{Steady Psi Sphere} for any rotation.
 
This paper is devoted to the study of stationary solutions that are close to the steady configurations presented above. More precisely, we will focus on two zonal base flows:
\begin{itemize}
\item The rigid rotation $\a Y_1^0=\a\frac{1}{2}\sqrt{\frac{3}{\pi}}\cos\theta$, of amplitude $\a\in\R$,
\item The zonal Rossby-Haurwitz stream function of second degree $\Psi_*:=\Psi_2=\b Y_2^0+\frac{\widetilde{\U}}{2}\cos\theta$.
\end{itemize}
Our choice of these two zonal flows is mainly motivated by the fact that the rigid rotation is the simplest flow motion one can have on a rotating sphere. In order to study the analogues of the Kolmogorov flow on the sphere, since the rigid rotation is itself a spherical harmonic on the first shelve of the eigenvalues, we consider the zonal spherical harmonic belonging to the second shelve of eigenvalues, thus the choice of $\Psi_*$.
In the sequel, taking $\U=\sqrt{\frac{\pi}{3}}\widetilde{\U}$, we write the zonal Rossby-Haurwitz stream function of second degree as $\Psi_*=\b Y_2^0+\U Y_1^0$, while now \eqref{Euler Sphere} becomes
\begin{equation}\label{Euler Sphere U}
\partial_t\O+U\cdot\nabla(\O-4\U Y_1^0)=0, \quad U=\nabla^\perp \Psi, \quad \D\Psi = \O.
\end{equation}
and \eqref{Steady F Sphere} is now written as
\begin{equation}\label{Steady F Sphere U}
\D\Psi-4\U Y_1^0 = F(\Psi).	
\end{equation}
In this paper we investigate the structure of the set of stationary states near the rigid rotation and the zonal Rossby-Haurwitz stream function of second degree because these steady solutions may constitute possible end-state configurations for the long-time dynamics of perturbations of the rigid rotation and the zonal Rossby-Haurwitz solution.  These sets of steady states close to the background zonal flows $\a Y_1^0$ and $\Psi_*$ depend on the rescaled angular velocity $\U$ and the amplitude parameters $\a$ and $\b$, respectively. Thus, our analysis keeps track of these quantities to understand up to which degree they influence the solutions.

In this direction, our first main result shows the existence of real analytic steady non-trivial (not a linear combination of spherical harmonics of eigenvalue $\lambda_2=6$) and non-zonal solutions to the $\U$-rotating Euler equations \eqref{Euler Sphere U} which are arbitrarily close to the stream function $\Psi_*=\b Y_2^0+\U Y_1^0$, for $(\b,\U)\neq (0,0)$, in the space of real analytic functions $C^\o(\Sph)$ endowed with the Gevrey norm 
\begin{equation*}
\V u \V_{\G_\lambda}^2:=\sum_{n\geq1}\sum_{|m|\leq n}\mu_n^2e^{2\lambda\mu_n^{1/2}}|u_n^m|^2,
\end{equation*}
where $\lambda_n=n(n+1)$ and $\mu_n=\lambda_n+1$. See Section \ref{function spaces} for more details on this norm.
\begin{theorem}\label{main thm 1}
For any $\b,\U\in\R$ such that $\b^2+\U^2>0$, there exist $c_{\U,\b}>0$ and $\e_0>0$ such that for any $0<\e\leq\e_0$ there exist analytic functions $\Psi_\e\in C^\o(\Sph)$ and $F_\e\in C^\o(\R)$ such that
\begin{equation*}
\D \Psi_\e-4\U Y_1^0=F_\e(\Psi_\e)
\end{equation*}
and
\begin{equation*}
\V \b Y_2^0+\U Y_1^0-\Psi_\e\V_{\G_\lambda(\Sph)}=O(\e),
\end{equation*}
with
\begin{equation*}
\begin{split}
\langle \Psi_\e,Y_6^2\rangle_{L^2(\Sph)}&=-\b^2\frac{1}{36}\frac{45}{11\pi\sqrt{182}}c_{\U,\b}\e^2+O(\e^3), \\
\langle \Psi_\e,Y_4^2\rangle_{L^2(\Sph)} &=-\frac{1}{14}\left( \frac{3\sqrt{3}}{154\pi}(11\U^2-5\b^2) -\b^2\frac{1}{7\pi}\frac{30\U^2}{7\U^2+15\b^2} \right)c_{\U,\b}\e^2 + O(\e^3).
\end{split}
\end{equation*}
In particular, $\Psi_\e$ is non-zonal and non-trivial.
\end{theorem}
The above problem is radically different when one considers the rigid rotation $\a Y_1^0$ as the base zonal flow. Indeed, the next result shows that, possibly up to a discrete set of rotations, any sufficiently smooth travelling wave solution to the Euler equation sufficiently close to $\a Y_1^0$ must be zonal. 
\begin{theorem}\label{main thm 2 rigidity Y_1^0}
Let $a=\frac{1}{2}\sqrt{\frac{3}{\pi}}$ and let $\a,c,\U\in\R$ such that $\a\neq 0$ and
\begin{equation}\label{spec cond}
2a(\a+2\U)\neq n(n+1)(\a a-c), \quad	\text{for all }n\geq 1.
\end{equation} 
Then, there exists $\e_0>0$ such that any travelling wave solution to the Euler Equation on the $\U$-rotating sphere  of the form $U=U(\theta,\varphi-ct)$ with associated vorticity $\O$ satisfying 
\begin{equation*}
\V \O +2\a Y_1^0\V_{H^4} \leq \e_0
\end{equation*}
must be zonal, that is, $U=U(\theta)\textbf{e}_\varphi$ and $\O = \O(\theta)$.
\end{theorem}
On the other hand, we exhibit two examples for which condition \eqref{spec cond} fails and the conclusion of Theorem \ref{main thm 2 rigidity Y_1^0} above is not true. Firstly, note that the condition does not hold for the wave velocity $c=0$, and amplitudes $\a=\U$, coinciding with the $\U$-rotation of the sphere. For this choice of parameters we have the following result, which is a direct consequence of Theorem \ref{main thm 1} for $\b=0$.
\begin{corollary}\label{main corollary 1}
Let $\U\in\R\setminus \lbrace 0\rbrace$. Then, there exist $\e_0>0$ such that for any $0<\e\leq\e_0$ there exist analytic non-zonal and non-trivial steady solutions $\Psi_\e$ to the $\U$-rotating Euler equations \eqref{Euler Sphere U} such that
\begin{equation*}
\V \U Y_1^0-\Psi_\e\V_{\G_\lambda(\Sph)}=O(\e).
\end{equation*}
\end{corollary}
Secondly, let us remark that the Euler equations for an inviscid incompressible fluid in a non-rotating sphere in vorticity form are
\begin{equation}\label{Euler non rot Sphere}
\partial_t\overline{\O}+\overline{U}\cdot\nabla\overline{\O}=0, \quad \overline{U}=\nabla^\perp \overline{\Psi}, \quad \D\overline{\Psi} = \overline{\O}.
\end{equation}
Here, we denote the vorticity by $\overline{\O}$ to distinguish it from the vorticity $\O$ that solves the Euler equations in a rotating sphere \eqref{Euler Sphere}, and similarly for the velocity $\overline{U}$ and stream-function $\overline{\Psi}$. By close inspection of \eqref{Euler Sphere} and \eqref{Euler non rot Sphere}, one can see that steady state solutions $\O(\theta,\varphi)$ of the Euler equations \eqref{Euler Sphere} on a sphere  rotating with velocity $\widetilde{\U}$ correspond to travelling wave solutions to the Euler equations in a sphere at rest \eqref{Euler non rot Sphere} of the form
\begin{equation}\label{travelling wave vorticity}
\overline{\O}(t,\theta,\varphi)=-2\widetilde{\U}\cos\theta+\O(\theta,\varphi-\widetilde{\U}t).
\end{equation}
Since $2\widetilde{\U}\cos\theta=4\U Y_1^0$ and $\D Y_1^0=-2Y_1^0$, we also have the following relation for the stream-functions $\D\overline{\Psi}=\overline{\O}$ and $\D\Psi=\O$, that is,
\begin{equation}\label{travelling wave stream-function}
\overline{\Psi}(t,\theta,\varphi)=2{\U}Y_1^0+\Psi(\theta,\varphi-\widetilde{\U}t).
\end{equation}
From this observation, we find a second example for which condition \eqref{spec cond} is not true and one can obtain non-trivial non-zonal solutions. For the choice of wave velocities $c=\widetilde{\U}$ and amplitudes $\alpha=3\U$, with the relation $\U=\sqrt{\frac{\pi}{3}}\widetilde{\U}$, we have the following.
\begin{corollary}\label{main corollary 2}
Let $\U\in\R\setminus \lbrace 0\rbrace$. Then, there exist $\e_0>0$ such that for any $0<\e\leq\e_0$ there exist analytic non-zonal and non-trivial travelling wave solutions $\overline{\Psi}_\e$ to the non-rotating Euler equation \eqref{Euler non rot Sphere} such that
\begin{equation*}
\V 3\U Y_1^0-\overline{\Psi}_\e\V_{\G_\lambda(\Sph)}=O(\e).
\end{equation*}
\end{corollary}
Indeed, from \eqref{travelling wave vorticity} and \eqref{travelling wave stream-function}, choosing $\overline{\Psi}_\e=2{\U}Y_1^0+\Psi_\e(\theta,\varphi-\widetilde{\U}t)$, where $\Psi_\e(\theta,\varphi)$ is the non-zonal steady solution to the rotating Euler equations \eqref{Euler Sphere U} given by Theorem \ref{main thm 1} for $\beta=0$, it is clear that $\overline{\O}_\e=\D\overline{\Psi}_\e$ is a non-trivial (non-zonal) travelling wave solution to the Euler equations \eqref{Euler non rot Sphere} such that $\overline{\Psi}_\e$ is $\e$-close to $3\U Y_1^0$ in the analytic Gevrey space $\G_\lambda(\Sph)$. 

We finish our discussion by presenting a setting for which even if condition \eqref{spec cond} fails, the conclusion of Theorem \ref{main thm 2 rigidity Y_1^0} may still hold if one is willing to assume further conditions on the solution. This is the purpose of the following result.
\begin{corollary}\label{main corollary 3}
Let $\a\in\R\setminus\lbrace 0\rbrace$. Then, there exists $\e_0>0$ such that any steady solution to the non-rotating Euler Equations \eqref{Euler non rot Sphere} whose vorticity $\overline{\O}$ satisfies
\begin{equation*}
\V \overline{\O} +2\a Y_1^0\V_{H^4} \leq \e_0, \quad \langle \overline{\O}, Y_1^m\rangle_{L^2(\Sph)}=0, \text{ for }|m|=1,
\end{equation*}
must be zonal, that is, $\overline{U}=\overline{U}(\theta)\textbf{e}_\varphi$ and $\overline{\O} = \overline{\O}(\theta)$.
\end{corollary}

\subsection{Perspectives}
Below we present a short literature review on properties of steady solutions to the Euler equations set in both a rotating sphere and planar domains. Afterwards, relate the results of this paper with the current state-of-the-art of the field.

\subsubsection{Euler equations on a 2D flat domain}
In 1907, Orr \cite{Orr} discovered a mixing mechanism that damps the non-shear component of inviscid Euler solutions close to the Couette flow in a channel. This mixing mechanism, currently known as \emph{vorticity mixing}, produces a key effects on the dynamics of the flow. The mixing of the vorticity produces \emph{inviscid damping}, a phenomenon for which the velocity is damped. Later, this vorticity mixing was also seen to be present in rotating flat domains such as $\R^2$ or $\T^2$.

In the past few years there has been a huge development on the theory of inviscid damping. Linear results for flows near Couette, stricly monotone, radial and Kolmogorov flows can be found in \cite{Bedrossian-CotiZelati-Vicol, CotiZelati-Zillinger, Jia, Lin-Yang-Zhu, Wei-Zhang-Zhao-20, Wei-Zhang-Zhu, Zillinger-15, Zillinger-17}. However, the complete dynamics are described by the full nonlinear problems, which are considerably much harder. See \cite{Bedrossian-Masmoudi, Bedrossian-Germain-Masmoudi, Ionescu-Jia-20b, Masmoudi-Zhao-20b, Gallay} for non-linear inviscid damping results for solutions near some shear and radial flows. 

The study on the local structure of steady solutions has also attracted recent attention, see \cite{Choffrut-SzekelyhidiJr, Choffrut-Sverak, Izosimov-Khesin, Nadirashvili}. The domain in which the motion takes place actually plays an important role in the geometrical properties of the stationary solutions. For example, Hamel and Nadirashvili  proved that in a strip where the velocity has no stagnation points, the flow must be shear, see \cite{Hamel-Nadirashvili-19b} and the references therein for this and similar results in the half-plane and radial domains. These results confer the idea that certain steady solutions adapt to the geometry and symmetries of the domain they occupy. Further statements in this direction are given in  \cite{Constantin-Drivas-Ginsberg, GomezSerrano-Park-Shi-Yao}. Similarly, a current line of research investigates if steady solutions nearby some background stationary state inherit geometrical properties of this background state. In the case where the fluid domain is a channel, this motivates the following definition: We say that a background shear flow $\u$ is
\begin{itemize}
\item  \emph{Rigid}, if all steady solutions sufficiently near $\u$ are themselves shear flows.
\item \emph{Flexible} if, on the contrary, there exists steady solutions arbitrarily close to $\u$ that are not shear flows.
\end{itemize}
These definitions highly depend on the metric used to measure the distances. Moreover, whether the flow is rigid or flexible has quite substantial dynamical consequences. For instance, flexibility of a shear flow $\u$ in a certain metric space automatically rules out the possibility of non-linear inviscid damping towards a nearby shear flow for all initial perturbations arbitrarily close to the base shear flow in that space. Indeed, one could take as an initial condition the non-shear stationary solution. Results on this geometrical property are available for some basic shear flows and highlight the role of the regularity: Lin and Zeng answered the rigidity/flexibility dichotomy for the Couette flow in the periodic channel in \cite{Lin-Zeng}, Castro and Lear proved similar results for Couette in a periodic strip in \cite{Castro-Lear} and partial answers for the Poiseuille and Kolmogorov flows were obtained by Coti Zelati, Elgindi and Widmayer in \cite{CotiZelati-Elgindi-Widmayer-20}. 

In \cite{Lin-Zeng}, the authors proved that in the periodic channel steady solutions whose vorticity is sufficiently close to that of Couette in $H^s$ with $s>3/2$ must be shear. They also constructed non-shear steady solutions whose vorticity is arbitrarily close to the Couette flow vorticity in $H^s$, with $s<3/2$, implying that inviscid damping to a shear flow is not true in low regularity. In \cite{Castro-Lear}, the authors obtained a new family of non-trivial (non-shear) and smooth travelling waves for the 2D Euler equation in a periodic strip, with the associated vorticity being arbitrarily close to Couette in $H^s$, with $s<3/2$, also denying the possibility of non-linear inviscid damping back to a shear flow. 

In \cite{CotiZelati-Elgindi-Widmayer-20} the authors showed rigidity of the Poiseuille flow in the periodic channel for solutions close in vorticity in $H^{5+}$ regularity and rigidity of the Kolmogorov flow on the rectangular torus $\T^2_\delta=[0,2\pi\delta)\times[0,2\pi)$ for solution near the Kolmogorov vorticity in  $H^{3+}$ regularity. However, the situation for the Kolmogorv flow on the square torus was shown to be completely different: the authors constructed non-trivial steady solutions (non-shear and not in the kernel of the linearized Euler equations) to the Euler equations arbitrarily near the Kolmogorov flow in analytic regularity via a fixed point argument. This, in turn, implies that the linear inviscid damping for solutions near the Kolmogorov flow on $\T^2$ obtained in \cite{Wei-Zhang-Zhao-20} cannot be extended to the non-linear level, no matter the regularity of the initial perturbation.
\subsubsection{Euler equations on a rotating sphere}
The Euler equations on the two dimensional sphere are used to model geophysical fluid dynamics. Some of its applications include meteorological predictions and the study of the motion of the atmosphere of Earth and other planets of the Solar system. A complete introduction to the theory of solutions to Navier-Stokes and Euler equations on the two dimensional unit sphere can be found in \cite{Skiba} and the references therein. See also \cite{Samavaki-Tuomela} and the references therein for a discussion regarding the Navier-Stokes and Euler equations on compact Riemannian manifolds.


Just like shear flows in the flat Euclidean setting, zonal flows are basic for understanding the long time dynamics of the equations. In this direction, non-linear Lyapunov stability for a class of stationary flows was shown in \cite{Caprino-Marchioro} and non-linear structural stability of solutions belonging to the second eigenspace of the Laplace-Beltrami operator was obtained in \cite{Wiro-Shepherd}. In \cite{Cheng-Mahalov} the authors proved that finite-time averages of solutions stay close to a subspace of zonal flows, the initial data being arbitrarily far away from that subspace, further supporting the idea that zonal flows are possible end-states for general initial configurations. These results were later extended in \cite{Taylor} to rotationally symmetric surfaces.

More recently, Constantin and Germain studied the Euler equation on a rotating sphere in \cite{Constantin-Germain}. There, they obtained that any solution $\Psi$ to \eqref{Steady F Sphere} for some $\U\in\R$ must be zonal if the corresponding non-linearity $F$ is such that $F'>-6$, and they further remark that this conditions is sharp by considering the Rossby-Haurwitz solution of degree 2.  The non-zonal solutions constructed in Theorem \ref{main thm 1} and Corollary \ref{main corollary 1} provide another non-trivial (they are not a Rossby-Haurwitz stream function) explicit example of this sharpness.

\subsubsection{Contributions of the paper and further insights}
Up to our knowledge, our results are the first to show geometrical properties on stationary solutions near zonal Rossby-Haurwitz stream functions of degree 2, the analogue of the Kolmogorov flow in the planar case. Indeed, Theorem \ref{main thm 2 rigidity Y_1^0} is a spherical version of the rigidity of the Kolmogorov flow in $\T^2_\delta$, while Theorem \ref{main thm 1} is in the same spirit as the flexibility of the Kolmogorov flow in $\T^2$.

A direct consequence of this flexibility is that inviscid damping towards a zonal flow for solutions to the Euler equation on the rotating sphere whose stream function is close to the Rossby-Haurwitz solution $\Psi_*=\b Y_2^0+\U Y_1^0$, for $\b \neq 0$, is not true, independently of the regularity assumptions and of the rotation $\U$. It is also worth remarking here that the steady solutions obtained in Theorem \ref{main thm 1} (and Corollary \ref{main corollary 1}) do not contradict the results in \cite{Cheng-Mahalov} mentioned above for large rotations $\U$ because, despite being non-zonal, their time-averages (that is, themselves) are $\e$ close to the zonal flows $\b Y_2^0+\U Y_1^0$ and $\a Y_1^0$, respectively.

While Theorem \ref{main thm 2 rigidity Y_1^0} and the series of results Corollary \ref{main corollary 1}, Corollary \ref{main corollary 2} and Corollary \ref{main corollary 3} provide some description of the set of  steady solutions to the Euler equations on a (rotating) sphere which are close to the rigid rotation generated by the stream-function $\a Y_1^0$, the complete picture is far from being understood. Precisely, for the Euler equations on a sphere at rest, the set of solutions close to the rigid rotation $\a Y_1^0$ is very rich: On one hand, from Corollary \ref{main corollary 2} there are non-trivial (non-zonal) travelling wave solutions arbitrarily close to the rigid rotation, while on the other hand, from Corollary \ref{main corollary 3} all steady states orthogonal to $Y_1^m$ for $|m|=1$ and sufficiently close to the rigid rotation are zonal. This plethora of phenomena hint at an even richer long-time behaviour of general perturbations of the rigid rotation under the dynamics of the Euler equations on a non-rotating sphere.

\subsection{Organization of the Article}
We begin by providing the basic concepts of differential geometry, spherical harmonics and functional spaces that we will repeatedly use throughout the article in Section \ref{Preliminaries}. Next, we prove Theorem \ref{main thm 1} in Section \ref{sec main thm 1} by first showing an analogous result for functions in $H^2$ and then upgrading up to analytic regularity. Finally, we show Theorem \ref{main thm 2 rigidity Y_1^0} in Section \ref{sec rigidity}.

\section{Preliminaries}\label{Preliminaries}
In this section we introduce essential concepts for the development of the manuscript. We begin by defining the main differential geometric notions, such as the set of spherical coordinates we will use and the differential operators we will consider throughout the paper. Afterwards, we give the precise definition and properties of the eigenfunctions of the Laplace-Beltrami operator in $\Sph$, the so-called spherical harmonics and finally we provide the definition of Sobolev and Gevrey spaces on the sphere.
\subsection{Basic definitions on differential geometry}
We usually parametrize the two dimensional unit sphere by the standard colatitude-longitude spherical coordinates
\begin{equation*}
(\theta,\varphi)\mapsto (\sin\theta\cos\varphi,\sin\theta\sin\varphi,\cos\theta), \quad (\theta,\varphi)\in (0,\pi)\times (0,2\pi),
\end{equation*}
which covers $\Sph$ except for the half circle $\lbrace \theta\in(0,\pi),\,\varphi=0\rbrace$ and so we complement it with the chart
\begin{equation}\label{equator chart}
(\widetilde{\theta},\widetilde{\varphi})\mapsto (-\sin\widetilde{\theta}\cos\widetilde{\varphi}, -\cos\widetilde{\theta}, -\sin\widetilde{\theta}\sin\widetilde{\varphi}), \quad (\widetilde{\theta},\widetilde{\varphi})\in (0,\pi)\times (0,2\pi),
\end{equation}
which removes the equatorial half circle $\lbrace \theta=0,\, \varphi\in (\frac{\pi}{2},\frac{3\pi}{2})\rbrace$, to obtain a smooth atlas for $\Sph$. In the sequel we mainly work with the usual colatitude-longitude parametrization and we explicitly remark when we work with the other chart. The Riemannian metric of $\Sph$ in the colatitude-longitude chart is given by
\begin{equation*}
g(\theta,\varphi)=\begin{pmatrix}1&0\\0&\sin^2\theta\end{pmatrix},
\end{equation*}
from which one obtains that the vectors
\begin{equation*}
\textbf{e}_\theta=\partial_\theta, \quad \textbf{e}_\varphi = \frac{1}{\sin\theta}\partial_\varphi,
\end{equation*} 
form an orthonormal basis of the tangent space $T_p\Sph$ at $p\in \Sph\setminus\lbrace N,\,S\rbrace$, where $N$ and $S$ denote the North and South poles of the sphere, respectively. Given a scalar function $f:\Sph\rightarrow \R$, the differential operators gradient and Laplace-Beltrami are given by
\begin{equation*}
\begin{split}
\nabla f=\partial_\theta f \textbf{e}_\theta +\frac{1}{\sin\theta}\partial_\varphi f \textbf{e}_\varphi,\quad \D f&=\frac{1}{\sin\theta}\partial_\theta(\sin\theta \partial_\theta f)+\frac{1}{\sin^2\theta}\partial^2_{\varphi}f.
\end{split}
\end{equation*}
We further define the $-\frac{\pi}{2}$-rotation $J$ in the tangent space given by $J\textbf{e}_\theta=-\textbf{e}_\varphi$ and $J\textbf{e}_\varphi=\textbf{e}_\theta$, from which we set $\nabla^\perp f:= J\nabla f$. Similarly, the divergence and rotational operators for a vector field $\u=u_\theta \textbf{e}_\theta+u_\varphi \textbf{e}_\varphi$ are defined by
\begin{equation*}
\begin{split}
\text{div}(\u)&=\frac{1}{\sin\theta}\partial_\theta(\sin\theta u_\theta) +\frac{1}{\sin\theta}\partial_\varphi u_\varphi,\quad \text{curl}(\u)=-\frac{1}{\sin\theta}\partial_\theta(\sin\theta u_\varphi) +\frac{1}{\sin\theta}\partial_\varphi u_\theta.
\end{split}
\end{equation*}
Additionally, the differential volume form is $\d \Sph=\sin\theta\, \d\theta\,\d\varphi$. We refer the reader to \cite{Cheng-Mahalov, Lee-2, Skiba} for a comprehensive treatment on the details of the differential geometric viewpoint of the Euler equation and \cite{Miura} for the deduction of the Euler equations in vorticity formulation posed in the two dimensional unit sphere.
\subsection{Spherical Harmonics} 
The real spherical harmonics of degree $n\geq 0$ are the eigenfunctions of the negative spherical Laplace-Beltrami operator $-\D$ corresponding to the eigenvalue $\lambda_n=n(n+1)$, see \cite{Atkinson-Han, Muller, Skiba}. Each eigenvalue $\lambda_n$ has multiplicity $2n+1$ for each $n\geq0$ and the corresponding real spherical harmonics are given by
\begin{equation*}
Y_n^m=Y_n^m(\theta,\varphi)=\left\{\begin{array}{lc} 
(-1)^m\sqrt{2}\sqrt{\dfrac{2n+1}{4\pi}\dfrac{(n-m)!}{(n+m)!}}P_n^{|m|}(\cos\theta)\sin(|m|\varphi),&\text{if }m<0,\\[5pt]
\sqrt{\dfrac{2n+1}{4\pi}}P_n(\cos\theta),&\text{if }m=0, \\[5pt]
(-1)^m\sqrt{2}\sqrt{\dfrac{2n+1}{4\pi}\dfrac{(n-m)!}{(n+m)!}}P_n^{m}(\cos\theta)\cos(m\varphi),&\text{if }m>0,
\end{array}\right.
\end{equation*}
where $m=0,\pm 1, \dots, \pm n$. For $n\geq0$, $P_n$ are the Legendre polynomials defined by
\begin{equation*}
P_n(s)=\frac{1}{2^nn!}\frac{\d^n}{\d s^n}(s^2-1)^n,\quad s\in(-1,1),
\end{equation*}
which are solutions to the eigenvalue problem
\begin{equation}\label{eig eq Leg pol}
\frac{\d}{\d s}\left( (1-s^2)\frac{\d}{\d s}\right) P_n(s)=-n(n+1)P_n(s).
\end{equation}
For $n>0$ and $0<|m|\leq n$ the associated Legendre functions $P_n^m$ are given by
\begin{equation*}
P_n^m(s)=(-1)^m(1-s^2)^{m/2}\dfrac{d^m}{ds^m}P_n(s).
\end{equation*}
The set of all spherical harmonics $\lbrace Y_n^m:n\geq 0,\,|m|\leq n\rbrace$ forms an orthonormal basis for $L^2(\Sph)$. That is, any $u\in L^2(\Sph)$ can be written as
\begin{equation*}
u=\sum_{n\geq0}\sum_{|m|\leq n}u_n^m Y_n^m, \quad u_n^m=\int_{\Sph}u\,Y_n^m\d\Sph,
\end{equation*}
where the convergence is in the $L^2(\Sph)$ sense. For each $n\geq 0$ we define the linear subspace $\textbf{Y}_n= \text{span} \lbrace Y_n^m : |m|\leq n\rbrace$, any $u\in \textbf{Y}_n$ is such that $\D u = -\lambda_n u$. In the sequel, we will mainly work with
\begin{equation*}
Y_1^0(\theta)=\frac{1}{2}\sqrt{\frac{3}{\pi}}\cos\theta,\quad Y_2^0(\theta)=\frac{1}{4}\sqrt{\frac{5}{\pi}}(3\cos^2\theta-1), \quad Y_2^2(\theta,\varphi)=\frac{1}{4}\sqrt{\frac{15}{\pi}}\sin^2\theta\cos(2\varphi).
\end{equation*}

\subsection{Function spaces on the unit sphere}\label{function spaces}
Let $A=-\D+1$ and $\mu_n:=\lambda_n+1$, for all $n\geq0$. For any $k\in \mathbb{N}$ we define the inhomogeneous Sobolev space $H^k(\Sph)$ using the domain of definition of $A^{k/2}$, see \cite{Cao-Rammaha-Titi-2, Skiba}. More precisely, we set
\begin{equation*}
H^k(\Sph):=\left\lbrace u\in L^2(\Sph)\,:\, u=\sum_{n\geq0}\sum_{|m|\leq n}u_n^mY_n^m, \; \sum_{n\geq0}\sum_{|m|\leq n}\mu_n^k|u_n^m|^2<\infty \right\rbrace
\end{equation*}
together with the norm 
\begin{equation*}
\V u \V_{H^k}^2:=\sum_{n\geq0}\sum_{|m|\leq n}\mu_n^k|u_n^m|^2.
\end{equation*}
We further define the homogeneous Sobolev space $\dot{H}^k(\Sph)$ as the completion of $C^\infty_0(\Sph)$, the space of smooth functions with zero average on the sphere, with respect to the norm
\begin{equation*}
\V u \V_{\dot{H}^k}^2:=\sum_{n\geq1}\sum_{|m|\leq n}\lambda_n^k|u_n^m|^2.
\end{equation*}
We also introduce the analytic Gevrey class of functions $\G^{k/2}_{\lambda}(\Sph)$, for $\lambda>0$ given by the domain of $A^{k/2}e^{\lambda A^{1/2}}$. Indeed, 
\begin{equation*}
\G^{k/2}_{\lambda}(\Sph):=\left\lbrace u\in L^2(\Sph)\,:\, u=\sum_{n\geq1}\sum_{|m|\leq n}u_n^mY_n^m, \; \sum_{n\geq1}\sum_{|m|\leq n}\mu_n^ke^{2\lambda\mu_n^{1/2}}|u_n^m|^2<\infty \right\rbrace,
\end{equation*}
and for $u\in\G^{k/2}_\lambda$ we consider the norm
\begin{equation}\label{def gevrey norm}
\V u \V_{\G^{k/2}_\lambda}^2:=\sum_{n\geq1}\sum_{|m|\leq n}\mu_n^ke^{2\lambda\mu_n^{1/2}}|u_n^m|^2.
\end{equation}
Note that the space of real analytic functions is such that $C^\o(\Sph)=\bigcup_{\lambda>0}\G^{k/2}_\lambda$, for any $k\geq0$, see \cite{John}. Moreover, we also have the following.
\begin{lemma}[\cite{Cao-Rammaha-Titi-2}]
For $\lambda\geq0$ and $k>3/2$, the Hilbert space $\G^{k/2}_\lambda$ is a topological algebra. More precisely, if $u,v\in\G^{k/2}_\lambda$, then $uv\in\G^{k/2}_\lambda$ and
\begin{equation}\label{algebra ineq}
\V uv \V_{\G^{k/2}_\lambda}\leq C_k \V u \V_{\G^{k/2}_\lambda} \V v \V_{\G^{k/2}_\lambda},
\end{equation}
with $C_k>0$ depending only on $k$.
\end{lemma} 
In the sequel, we take $k=2$ and denote $\G_\lambda:=\G^1_\lambda$, a topological algebra thanks to the above lemma. Finally, for any $u,v\in L^2$ we consider the usual scalar product
\begin{equation*}
\langle u, v \rangle = \int_{\Sph} uv\,\d\Sph.
\end{equation*}

\section{Non-zonal stationary solutions near Rossby-Haurwitz}\label{sec main thm 1}
This section is devoted to the proof of Theorem \ref{main thm 1}. We begin by showing an analogous statement for stream functions belonging to $H^2$, see Theorem \ref{thm existence} below. Afterwards, we will use elliptic regularity theory to improve the regularity of the constructed steady solutions up to real analytic, see Theorem \ref{thm reg} below and section \ref{sec reg}. We follow the strategy presented in \cite{CotiZelati-Elgindi-Widmayer-20} to prove the existence of non-trivial non-shear solutions to the Euler equation arbitrarily close to the Kolmogorov flow on the square torus.

The Euler equation on a rotating sphere \eqref{Euler Sphere U} for a small steady perturbation $\o$ around the vorticity $\O_*=-6\b Y_2^0-2\U Y_1^0$ associated to the Rossby-Haurwitz stream function $\Psi_*=\b Y_2^0+\U Y_1^0$ reads
\begin{equation*}
\begin{split}
0&=\partial_t\o+\frac{1}{2}\left( 3\b\sqrt{\frac{5}{\pi}}\cos\theta+\U\sqrt{\frac{3}{\pi}}\right)\left(1+6\D^{-1}\right)\partial_\varphi\o + u\cdot\nabla\o, 
\end{split}
\end{equation*}
from which we define the linearized operator 
\begin{equation*}
\L\o:=\frac{1}{2}\left( 3\b\sqrt{\frac{5}{\pi}}\cos\theta+\U\sqrt{\frac{3}{\pi}}\right)\left(1+6\D^{-1}\right)\partial_\varphi\o.
\end{equation*}
A close inspection shows that the kernel of $\L$ is formed by zonal flows and eigenfunctions of the Laplace-Beltrami operator whose eigenvalue is $-6$, that is, the subspace $\textbf{Y}_2$. This motivates our definition of non-trivial solutions as those that are not in the kernel of the linearized operator. Indeed, $\Psi_\e=\b Y_2^0+\U Y_1^0 + \e\textbf{Y}_2$ is already a non-zonal stationary solution $\e$ close to the Rossby-Haurwitz stream function, and so we are interested in more general non-zonal solutions.

The following result shows the existence of steady non-trivial and non-zonal solutions to the rotating Euler equation on $\Sph$ arbitrarily close to $\b Y_1^0+\U Y_1^0$ in $H^2$. 
\begin{theorem}\label{thm existence}
For all $\b,\U\in \R$ such that $\b^2+\U^2>0$, there exist $c_{\U,\b}>0$ and $\e_0>0$ such that for any $0<\e\leq \e_0$ there exist functions $\Psi_\e\in H^2(\Sph)$ and $F_\e:\R\rightarrow \R$ for which
\begin{equation*}
\D \Psi_\e-4\U Y_1^0=F_\e(\Psi_\e)
\end{equation*}
and
\begin{equation*}
\V \b Y_2^0+\U Y_1^0-\Psi_\e\V_{H^2}=O(\e),
\end{equation*}
with
\begin{equation*}
\begin{split}
\langle \Psi_\e,Y_6^2\rangle&=-\b^2\frac{1}{36}\frac{45}{11\pi\sqrt{182}}c_{\U,\b}\e^2+O(\e^3), \\
\langle \Psi_\e,Y_4^2\rangle &=-\frac{1}{14}\left( \frac{3\sqrt{3}}{154\pi}(11\U^2-5\b^2) -\b^2\frac{1}{7\pi}\frac{30\U^2}{7\U^2+15\b^2} \right)c_{\U,\b}\e^2 + O(\e^3).
\end{split}
\end{equation*}
\end{theorem}
On the other hand, the next result asserts that the functions obtained above are smoother than just $H^2(\Sph)$, in fact they are real analytic. Furthermore, the following result shows that these new functions $\Psi_\e$ are arbitrarily close to $\b Y_2^0+\U Y_1^0$ in the Gevrey class of functions $\G_\lambda$, for some $\lambda>0$.
\begin{theorem}\label{thm reg}
The solution $\Psi_\e\in H^2$ given by Theorem \ref{thm existence} is real analytic, that is, $\Psi_\e\in C^\o(\Sph)$. Moreover, there exists $\lambda>0$ and $M>0$, both independent of $\e>0$ such that
\begin{equation*}
\V \b Y_2^0 +\U Y_1^0 - \Psi_\e \V_{\G_\lambda}\leq M\e, \text{ for all }0<\e\leq \e_0.
\end{equation*}
\end{theorem}
\subsection{Proof of Theorem \ref{thm existence}}
We show Theorem \ref{thm existence} in the following subsections. We begin in \ref{Setup} by setting up the basis that will lead to a contraction argument in \ref{Contraction}. Thereafter, in \ref{Non-triviality} we use the explicit expression of our functions to deduce the presence of non-trivial modes. This guarantees that the associated flows are non-zonal and neither in $\textbf{Y}_2$.
\subsubsection{Setup of the argument}\label{Setup} The existence of a stationary state near $\Psi_*=\b Y_2^0 + \U Y_1^0$ relies on constructing the solution $\Psi_\e$ perturbatively from the stream function $\Psi_*$ and is based on the fact that $\Psi_*$ satisfies $\Delta \Psi_*-4\U Y_1^0=F_*(\Psi_*)$ with $F_*(s)=-6s$. It is then natural to make the ansatz
\begin{equation*}
\Psi_\e=\Psi_*+\e\psi,\quad F_\e=F_*+\e f,
\end{equation*}
which produces a nonlinear elliptic equation for $\psi$, with $f$ to be determined as well,
\begin{equation*}
\D \psi + 6\psi = f(\Psi_*+\e\psi).
\end{equation*}
Note that the operator $\D+6$ is, in general, not invertible, and while this constitutes some difficulties, it also permits to introduce via $\psi$ elements of the kernel $\ker \L$, for instance, $Y_2^2$, which gives the following equation to be solved
\begin{equation}\label{eq psi Y22}
\D\psi+6\psi=f(\Psi_*+\e\psi+\e Y_2^2),
\end{equation}
with $\psi\perp\ker(\D+6)$. We choose $f$ as a cubic polynomial $f=f(z)=Az+Bz^2+c_{\U,\b} z^3$ with coefficients $A,B\in\R$, and $c_{\U,\b}>0$, where $A,B$ will be determined as functionals of $\psi$ and $\e>0$ and $c_{\U,\b}$ will be chosen later on. We obtain
\begin{equation}\label{eq f}
\begin{split}
\D\psi +6\psi=&A\P+B\P^2+c_{\U,\b}\P^3\\
&+\e(\psi+Y_2^2)\big( A+2B\P+3c_{\U,\b}\P^2\big)\\
&+R(B,c_{\U,\b},\psi,\e),
\end{split}
\end{equation}
where
\begin{equation*}
R(B,c_{\U,\b},\psi,\e)=\e^2(\psi+Y_2^2)^2\big(B+3c_{\U,\b} \P\big)+\e^3c_{\U,\b}(\psi+Y_2^2)^3.
\end{equation*}
In order to remark the dependence of the polynomial $f$ on the coefficients $A,\,B$ and $c_{\gamma,\beta}$, we denote $f=f(A,B,c_{\gamma,\beta};z)$. A necessary condition to have solutions to the semilinear elliptic problem \eqref{eq psi Y22} is that the right hand side of \eqref{eq f} must be orthogonal to the kernel $\ker (\D+6)=\textbf{Y}_2$, so $f(A,B,c_{\U,\b};\P+\e Y_2^2 +\e\psi)$ must satisfy five orthogonality conditions.

We can reduce the number of orthogonality conditions by exploiting the symmetries of the spherical harmonics and the polynomial nature of the nonlinearity $f$. Hence, if we assume that $\psi$ is an even function  in $\varphi$, that is, $\langle\psi, Y_n^{-|m|}\rangle=0$, for all $n\geq1$ and all $0<|m|\leq n$, it is then straightforward to see that $f(A,B,c_{\U,\b};\P+\e Y_2^2 +\e\psi)$ is also even in $\varphi$ for all $A,B,c_{\U,\b}\in\R$, so that the two orthogonality conditions regarding $Y_2^{-2}$ and $Y_2^{-1}$ are automatically satisfied. 

If we also further assume that $\langle\psi, Y_n^{2k+1}\rangle=0$, for all $n\geq0$ and all $k\geq 0$ such that $2k+1\leq n$, this in turn implies that $f(A,B,c_{\U,\b};\P+\e Y_2^2 +\e\psi)$ is orthogonal to $Y_2^1$. Indeed, any multiplication of spherical harmonics of the form $Y_{n_1}^{2m_1}Y_{n_2}^{2m_2}$, for $0\leq 2m_i\leq n_i$ cannot generate spherical harmonics of the form $Y_{n_3}^{2m_3+1}$, for any $0\leq 2m_3+1\leq n_3$. The other two orthogonality conditions which remain to be considered are
\begin{equation}\label{compt cond}
\langle f(A,B,c_{\U,\b};\P+\e Y_2^2 +\e\psi),Y_2^0\rangle=\langle f(A,B,c_{\U,\b};\P+\e Y_2^2 +\e\psi),Y_2^2\rangle=0.
\end{equation}
These equations are restrictions for the coefficients $A=A(\psi;\e)$ and $B=B(\psi;\e)$. Using the given expression for $f$, we find that
\begin{equation}\label{cond A B}
\begin{split}
0=\b A(\psi;\e)&+\frac{7\U^2+5\b^2}{7\sqrt{5\pi}}B(\psi;\e)+3\b\frac{11\U^2+5\b^2}{28\pi}c_{\U,\b} \\
&+\e\left\langle (2B(\psi;\e)\P+3c_{\U,\b}\P^2)\psi,Y_2^0\right\rangle + \langle R,Y_2^0\rangle, \\
0=A(\psi;\e) &-\left(\b\frac{2}{7}\sqrt{\frac{5}{\pi}}-2\langle \P\psi,Y_2^2\rangle\right)B(\psi;\e)+3\left(\frac{3\U^2+5\b^2}{28\pi}+\langle \P^2\psi, Y_2^2\rangle\right)c_{\U,\b} \\
&+ \frac{1}{\e}\langle R, Y_2^2 \rangle.
\end{split}
\end{equation}
From here we can easily obtain
\begin{equation}\label{eq B beta}
\begin{split}
0&=\left(\frac{7\U^2+15\b^2}{35}\sqrt{\frac{5}{\pi}}-2\b\langle \P\psi, Y_2^2\rangle\right)B(\psi;\e)+\b\left(\frac{6\U^2}{7\pi}-3\langle \P^2\psi,Y_2^2\rangle\right) c_{\U,\b} \\
&+\e\left\langle (2B(\psi;\e)\P+3c_{\U,\b}\P^2)\psi,Y_2^0\right\rangle+ \langle R,Y_2^0\rangle-\b\frac{1}{\e}\langle R,Y_2^2\rangle.
\end{split}
\end{equation}
Together with the above assumptions on the orthogonality of $\psi$ with certain families of spherical harmonics and in order to be able to solve for $B(\psi;\e)$ in \eqref{eq B beta}, we define the function space $X$ we will work in by
\begin{equation}\label{def X}
\begin{split}
X&:=\left\lbrace \psi\in H^2: \psi\perp Y_2^m, Y_n^{-|k|}, Y_n^{2k-1},\quad |m|\leq 2,\, n>0,\, k>0, \right. \\
&\qquad\left.|\langle\psi,\P Y_2^2\rangle|\leq\frac{1}{8}\frac{\U^2+\b^2}{\sqrt{5\pi}},\quad |\langle\psi,\P^2Y_2^2\rangle|\leq\frac{\U^2+\b^2}{3},\quad \V \psi\V_{H^2}\leq 150(1+\U^2+\b^2)^2 \right\rbrace. 
\end{split}
\end{equation}
The following lemma determines $A(\psi;\e)$ and $B(\psi;\e)$ and state some of their properties. 
\begin{lemma}\label{lemma def A B}
There exists $\e_1>0$ such that for $\psi\in X$ and for $0\leq \e\leq \e_1$ the relations \eqref{cond A B} recursively define real sequences $(a_j(\psi))_{j\geq0}$ and $(b_j(\psi))_{j\geq0}$ both depending on $c_{\gamma,\beta}$ and such that
\begin{equation}\label{def A B}
A(\psi;\e):=\sum_{j\geq0}a_j(\psi)\e^j, \qquad B(\psi;\e):=\sum_{j\geq0}b_j(\psi)\e^j
\end{equation}
are well-defined, uniformly bounded for $\psi\in X$ and satisfy \eqref{cond A B}. Moreover, the maps
\begin{equation*}
\psi\mapsto a_j(\psi),\qquad \psi\mapsto b_j(\psi),\quad j\geq0,
\end{equation*}
are Lipschitz continuous on $L^2(\Sph)$ with constants $L_j\leq L^j$, for some $L>0$, and the maps
\begin{equation*}
\psi\mapsto a_0(\psi), \qquad \psi\mapsto b_j(\psi)
\end{equation*}
are Lipschitz continuous on $\dot{H}^2(\Sph)$, with 
\begin{equation*}
\begin{split}
|a_0(\psi_1)-a_0(\psi_2)| &< \frac{1}{2}(1+|\b|)(\U^2+\b^2)c_{\U,\b}\V \psi_1-\psi_2\V_{\dot{H}^2},\\
|b_0(\psi_1)-b_0(\psi_2)|&<\frac{2}{25}|\b|c_{\U,\b}\V \psi_1-\psi_2\V_{\dot{H}^2}.
\end{split}
\end{equation*}
\end{lemma}
\begin{remark}
The Lipschitz continuity of the coefficients $a_0(\psi)$ and $b_0(\psi)$ for $\psi\in X\subset H^2(\Sph)$ is key for proving the contraction argument for the subspace $X$ that will guarantee the existence and uniqueness of solution.
\end{remark}
\begin{proof}
From \eqref{eq B beta} one can write,
\begin{equation}\label{eq B}
\begin{split}
\left(\frac{7\U^2+15\b^2}{35}\sqrt{\frac{5}{\pi}}-2\b\langle \P\psi, Y_2^2\rangle\right)B(\psi;\e)&=-\b\left(\frac{6\U^2}{7\pi}-3\langle \P^2\psi,Y_2^2\rangle\right) c_{\U,\b} \\
&-\e\left\langle (2B(\psi;\e)\P+3c_{\U,\b}\P^2)\psi,Y_2^0\right\rangle \\
&- \langle R,Y_2^0\rangle+\frac{\b}{\e}\langle R,Y_2^2\rangle.
\end{split}
\end{equation}
Expanding here $B(\psi,\e)$ in series as $B(\psi;\e):=\sum_{j\geq0}b_j(\psi)\e^j$ and comparing coefficients in $\e$ one shows that $b_j(\psi)$ can be inductively defined from a linear combination of $b_{j-2}(\psi)$ and $b_{j-1}(\psi)$. The coefficients of this linear combination are $L^2(\Sph)$ inner products of $\psi$ against  spherical harmonics. These inner products are uniformly bounded because $\psi\in X$, see \eqref{def X}. Therefore, one can show by induction that there exists $M$ sufficiently large such that $|b_j(\psi)|\leq M^j$ and the series expansion for $B(\psi,\e)$ converges for $0\leq \e< M^{-1}$. The same holds for $A(\psi,\e):=\sum_{j\geq0}a_j(\psi)\e^j$ because we can use 
\begin{equation*}
A(\psi;\e) =\left(\b\frac{2}{7}\sqrt{\frac{5}{\pi}}-2\langle \P\psi,Y_2^2\rangle\right)B(\psi;\e)-3\left(\frac{3\U^2+5\b^2}{28\pi}+\langle \P^2\psi, Y_2^2\rangle\right)c_{\U,\b} - \frac{1}{\e}\langle R, Y_2^2 \rangle.
\end{equation*}
to find $a_j(\psi),\,j\geq0$, directly from $B(\psi,\e)$, with
\begin{equation}\label{def a0}
a_0(\psi)=\left(\b\frac{2}{7}\sqrt{\frac{5}{\pi}}-2\langle \P\psi,Y_2^2\rangle\right)b_0(\psi)-3\left(\frac{3\U^2+5\b^2}{28\pi}+\langle \P^2\psi, Y_2^2\rangle\right)c_{\U,\b}. 
\end{equation}
The maps $\psi\mapsto a_j(\psi)$ and $\psi\mapsto b_j(\psi)$ for $j\geq0$ are Lipschitz. For $j\geq1$ it follows from the recursive construction of the coefficients, while for $j=0$ we observe that
\begin{equation}\label{def b0}
b_0(\psi)=B(\psi;0)=-\b c_{\U,\b}\frac{\left(\frac{6\U^2}{7\pi}-3\langle \P^2\psi,Y_2^2\rangle\right)}{\frac{7\U^2+15\b^2}{35}\sqrt{\frac{5}{\pi}}-2\b\langle \P\psi, Y_2^2\rangle}=:-\b c_{\U,\b}\frac{n(\psi)}{d(\psi)},
\end{equation}
which is well-defined because $\psi\in X$ implies $d(\psi)\neq 0$, see \eqref{def X}. We compute
\begin{equation*}
\P Y_2^2=(\b Y_2^0+\U Y_1^0)Y_2^2=\b\frac{1}{14}\sqrt{\frac{15}{\pi}}Y_4^2+\frac{\U}{2}\sqrt{\frac{3}{7\pi}}Y_3^2-\b\frac{1}{7}\sqrt{\frac{5}{\pi}}Y_2^2
\end{equation*}
and
\begin{equation*}
\begin{split}
\P^2Y_2^2&=(\b Y_2^0+\U Y_1^0)^2Y_2^2 \\
&=\b^2\frac{15}{11\pi\sqrt{182}}Y_6^2+\b\frac{\U}{2\pi}\sqrt{\frac{15}{77}}Y_5^2+\frac{\sqrt{3}}{154\pi}(11\U^2-5\b^2)Y_4^2+\frac{3\U^2+5\b^2}{28\pi}Y_2^2.
\end{split}
\end{equation*}
Recalling that $\langle Y_2^m,\psi \rangle =0$ and $\V Y_n^m \V_{\dot{H}^{-2}(\Sph)}=\frac{1}{n(n+1)}$, we have
\begin{equation*}
\begin{split}
|n(\psi_1)-n(\psi_2)| &\leq 3 \left( \b^2\frac{15}{11\pi\sqrt{182}}\frac{1}{42}+|\b|\frac{|\U|}{2\pi}\sqrt{\frac{15}{77}}\frac{1}{30}+\frac{\sqrt{3}}{154\pi}(11\U^2+5\b^2)\frac{1}{20}\right) \V\psi_1-\psi_2\V_{\dot{H}^2} \\
&<\frac{3}{250}(\b^2+\U^2)\V\psi_1-\psi_2\V_{\dot{H}^2},
\end{split}
\end{equation*}
and
\begin{equation*}
\begin{split}
|d(\psi_1)-d(\psi_2)| &\leq 2|\b|\left(|\b|\frac{1}{14}\sqrt{\frac{15}{\pi}}\frac{1}{20}+\frac{|\U|}{2}\sqrt{\frac{3}{7\pi}}\frac{1}{12}\right) \V\psi_1-\psi_2\V_{\dot{H}^2} \\
&< \frac{6}{125}(\b^2+\U^2)\V\psi_1-\psi_2\V_{\dot{H}^2}.
\end{split}
\end{equation*}
Then, 
\begin{equation*}
\begin{split}
|b_0(\psi_1)-b_0(\psi_2)|&= |\b|c_{\U,\b}\left| \frac{n(\psi_1)d(\psi_2)-n(\psi_2)d(\psi_1)}{d(\psi_1)d(\psi_2)} \right| \\
&\leq |\b|c_{\U,\b}\left| \frac{ n(\psi_1)-n(\psi_2)}{d(\psi_1)}\right| + |\b|c_{\U,\b}\left| \frac{n(\psi_2)\left(d(\psi_2)-d(\psi_1)\right)}{d(\psi_1)d(\psi_2)}\right|,
\end{split}
\end{equation*}
and we deduce that the $\dot{H}^2$ (and also the $H^2$) Lipschitz constant of $b_0$ is bounded by
\begin{equation*}
\begin{split}
|\b|&c_{\U,\b}\frac{\left(\frac{7\U^2+15\b^2}{35}\sqrt{\frac{5}{\pi}} +\frac{1}{4}\frac{\U^2+\b^2}{\sqrt{5\pi}}\right)\frac{3}{250}(\b^2+\U^2)+\left(\frac{6\U^2}{7\pi}+\U^2+\b^2\right)\frac{6}{125}(\b^2+\U^2)}{\left(\frac{7\U^2+15\b^2}{35}\sqrt{\frac{5}{\pi}}-\frac{1}{4}\frac{\U^2+\b^2}{\sqrt{5\pi}}\right)^2} < \frac{2}{25}|\b|c_{\U,\b}.
\end{split}
\end{equation*}
For the Lipschitz constant of $a_0$ we have 
\begin{equation*}
\begin{split}
|a_0(\psi_1)-a_0(\psi_2)|&\leq \left( \frac{2}{7}\sqrt{\frac{5}{\pi}}|\b|+2|\langle \P\psi_1,Y_2^2\rangle|\right)|b_0(\psi_1)-b_0(\psi_2)| \\
&+2|\langle \P(\psi_1-\psi_2),Y_2^2\rangle||b_0(\psi_2)| + 3c_{\U,\b}|\langle \P^2(\psi_1-\psi_2),Y_2^2\rangle|,
\end{split}
\end{equation*}
thanks to \eqref{def a0}. From \eqref{def b0} we can bound
\begin{equation*}
|b_0(\psi)| \leq \frac{\frac{6\U^2}{7\pi}+\U^2+\b^2}{\frac{7\U^2+15\b^2}{35}\sqrt{\frac{5}{\pi}}-\frac{1}{4}\frac{\U^2+\b^2}{\sqrt{5\pi}}}|\b|c_{\U,\b} < 8|\b|c_{\U,\b},
\end{equation*}
and we also have 
\begin{equation*}
\begin{split}
2|\langle \P(\psi_1-\psi_2),Y_2^2\rangle|&\leq \frac{4}{125}(|\b|+|\U|)\V\psi_1-\psi_2\V_{\dot{H}^2}, \\
3|\langle \P^2(\psi_1-\psi_2),Y_2^2\rangle|&\leq \frac{3}{250}(\b^2+\U^2)\V\psi_1-\psi_2\V_{\dot{H}^2}.
\end{split}
\end{equation*}
Combining these estimates we obtain
\begin{equation*}
|a_0(\psi_1)-a_0(\psi_2)| < \frac{1}{2}(1+|\b|)(\U^2+\b^2)c_{\U,\b}\V \psi_1-\psi_2\V_{\dot{H}^2},
\end{equation*}
the proof is complete.
\end{proof}
\begin{lemma}\label{lemma prop A B}
Let $\psi,\psi_j\in X$, $j\in {1,2}$ and let $A(\psi;\e)$ and $B(\psi;\e)$ as in Lemma \ref{lemma def A B}. Then, for $\e>0$ sufficiently small we have that
\begin{equation}\label{bounds A B}
|A(\psi,\e)|\leq 5(1+|\b|)(\U^2+\b^2)c_{\U,\b},\quad |B(\psi,\e)|\leq 8(1+|\b|)c_{\U,\b},
\end{equation}
\begin{equation}\label{Lip bounds A B}
\begin{split}
|A(\psi_1;\e)-A(\psi_2;\e)| &< \frac{3}{4}(1+|\b|)(\U^2+\b^2)c_{\U,\b}\V \psi_1-\psi_2\V_{\dot{H}^2},\\
|B(\psi_1;\e)-B(\psi_2;\e)|&<\frac{1}{5}|\b|c_{\U,\b}\V \psi_1-\psi_2\V_{\dot{H}^2},
\end{split}
\end{equation}
and
\begin{equation}\label{bound R}
\V R(B(\psi;\e),c_{\U,\b},\psi,\e)\V_{L^2} \lesssim \e^2,
\end{equation}
\begin{equation}\label{Lip bound R}
\V R(B(\psi_1;\e),c_{\U,\b},\psi_1,\e)-R(B(\psi_2;\e),\psi_2,\e)\V_{L^2} \lesssim \e^2\V \psi_1-\psi_2\V_{L^2}.
\end{equation}
\end{lemma}
\begin{proof}
We begin by showing \eqref{bounds A B} for $B(\psi;\e)$. We already know $|b_0(\psi)|<8|\b|c_{\U,\b}$ and from the proof of Lemma \ref{lemma def A B} we can find $M$ large enough such that $|b_j(\psi)|\leq M^j$, so that for $\e>0$ small enough one has $\sum_{j\geq 1}(M\e)^j=\frac{M\e}{1-M\e}<8c_{\U,\b}$ and the results swiftly follows.
Similarly, from \eqref{def a0} and the bound on $|b_0(\psi)|$ we deduce that
\begin{equation*}
\begin{split} 
|a_0(\psi)| &\leq \left(\frac{2}{7}\sqrt{\frac{5}{\pi}}|\b|+\frac{1}{4}\frac{\U^2+\b^2}{\sqrt{5\pi}}\right)8|\b|c_{\U,\b} +\left(3\frac{3\U^2+5\b^2}{28\pi}+\U^2+\b^2\right)c_{\U,\b} \\
&<\frac{9}{2}(1+|\b|)(\U^2+\b^2)c_{\U,\b}
\end{split}
\end{equation*}
and we control $\sum_{j\geq1}|a_j(\psi)|\e^j\leq \frac{1}{2}(1+|\b|)(\U^2+\b^2)c_{\U,\b}$ for $\e>0$ small enough as we did above.
The bounds \eqref{Lip bounds A B} follow from the definition of both $A(\psi;\e)$ and $B(\psi;\e)$ in \eqref{def A B} and the Lipschitz constants of $a_0(\psi)$ and $b_0(\psi)$ from Lemma \ref{lemma def A B}. Finally, the bounds \eqref{bound R} and \eqref{Lip bound R} are easily deduced from the definition of $R(B(\psi;\e),\psi,\e;\cdot,\cdot)$.
\end{proof}
\subsubsection{Contraction Mapping}\label{Contraction}
The solutions to \eqref{eq f} are constructed as fixed points of a map on $X$. Therefore, we define
\begin{equation*}
K_\e:X\rightarrow H^2,\quad \psi\mapsto K_\e(\psi),
\end{equation*}
where $K_\e(\psi)$ solves $(\D+6)K_\e=f(A(\psi;\e),B(\psi;\e),c_{\U,\b};\P+\e Y_2^2+\e\psi)$ in the orthogonal complement of $\textbf{Y}_2$, where $A(\psi;\e)$ and $B(\psi;\e)$ are defined as above, guaranteeing that the necessary orthogonality conditions are satisfied. 
\begin{proposition}\label{Prop Contraction}
$K_\e$ defines a contraction on $(X,\V\cdot\V_{H^2})$, for $\e>0$ small enough.
\end{proposition}
\begin{proof}
Let $0<\e<\e_1$ for which by Lemma \ref{lemma def A B} the coefficients $A(\psi;\e)$ and $B(\psi;\e)$ are well-defined. We first show that $K_\e$ maps $X$ into itself, see \eqref{def X} for the definition of $X$. By construction, since $\psi$ is orthogonal to $Y_n^{-|k|}$ and $Y_n^{2k-1}$ for all $n,k>0$ and $f$ is a cubic polynomial, it is clear that $K_\e(\psi)$ will be orthogonal to $Y_n^{-|k|}$ and $Y_n^{2k-1}$ for all $n,k>0$, too. Moreover, it is straightforward to see that 
\begin{equation*}
|\langle K_\e(\psi),\P Y_2^2\rangle|+|\langle K_\e(\psi),\P^2Y_2^2\rangle|\lesssim \e,
\end{equation*}
since there are no non-zonal spherical harmonics at order 0 in $\e$ in the expansion of the nonlinearity $f(A(\psi;\e),B(\psi;\e),c_{\U,\b};\P+\e Y_2^2+\e\psi)$ in orders of $\e$. Now, since $\psi\in X$ we have $\V \psi \V_{L^\infty}\leq \V \psi \V_{H^2}\leq 150(1+\U^2+\b^2)^2$. Together with the bounds \eqref{bounds A B} and $\V \Psi_*\V_{L^\infty}\leq (|\b|+|\U|)$, one can easily prove that 
\begin{equation*}
\begin{split}
|f(A(\psi;\e),&B(\psi;\e),c_{\U,\b};\P+\e Y_2^2+\e\psi)|\\
&\leq |A(\psi;\e)|\V \P\V_{L^\infty} + |B(\psi;\e)|\V \P\V_{L^\infty}^2
+ |c_{\U,\b}|\V \P\V_{L^\infty}^3+O(\e) \\
&\leq 5(1+|\b|)(\U^2+\b^2)(|\b|+|\U|)c_{\U,\b} + 8(1+|\b|)(|\b|+|\U|)^2c_{\U,\b}  \\
&\quad + (|\b|+|\U|)^3c_{\U,\b} + O(\e) \\
&\leq 17(1+\U^2+\b^2)^2 + O(\e),
\end{split}
\end{equation*}
which in turn gives
\begin{equation*}
\V f(A(\psi;\e),B(\psi;\e);\P+\e Y_2^2+\e\psi) \V_{L^2}\leq 65(1+\U^2+\b^2)^2,
\end{equation*}
for $\e>0$ small enough. For
\begin{equation*}
\widetilde{K_\e}:\psi \mapsto f(A(\psi;\e),B(\psi;\e),c_{\U,\b};\P+\e Y_2^2+\e\psi).
\end{equation*}
we have that $K_\e(\psi)=(\D+6)^{-1}\widetilde{K_\e}(\psi)$ and since $\langle \widetilde{K_\e}(\psi),Y_2^m\rangle=0$ for all $0\leq |m|\leq 2$, 
\begin{equation*}
\V K_\e(\psi)\V_{H^2}\leq \frac{13}{6}\V \widetilde{K_\e}(\psi)\V_{L^2}\leq 150(1+\U^2+\b^2)^2.
\end{equation*}
Thus, for $\e$ small enough we conclude that $K_\e(X)\subset X$. We finish the proof of the proposition by showing that $K_\e$ is indeed a contraction mapping. For this, let $\psi_1,\psi_2\in X$ and define $G_j=\P+\e Y_2^2 +\e\psi_j$ for $j=1,\,2$. Observe that
\begin{equation*}
\begin{split}
\widetilde{K_\e}(\psi_1)-\widetilde{K_\e}(\psi_2)&=f(A(\psi_1;\e),B(\psi_1;\e),c_{\U,\b};G_1)-f(A(\psi_2;\e),B(\psi_2;\e),c_{\U,\b};G_2)\\
&=(A(\psi_1;\e)-A(\psi_2;\e))\P+(B(\psi_1;\e)-B(\psi_2;\e))\P^2 +O(\e)\\
&=(a_0(\psi_1)-a_0(\psi_2))\P+(b_0(\psi_1)-b_0(\psi_2))\P^2 + O(\e)
\end{split}
\end{equation*}
Thus, up to terms of order $\e$ we use the Lipschitz constants for $a_0(\psi)$ and $b_0(\psi)$ in $\dot{H}^2(\Sph)$ previously found and we bound
\begin{equation*}
\begin{split}
\V (a_0(\psi_1)&-a_0(\psi_2))\P+(b_0(\psi_1)-b_0(\psi_2))\P^2 \V_{L^2}\\
& \leq \left( \frac{1}{2}(1+|\b|)(\U^2+\b^2)\V \P\V_{L^2} + \frac{2}{25}|\b|\V \P^2\V_{L^2}\right) c_{\U,\b}\V \psi_1-\psi_2\V_{\dot{H}^2}.
\end{split}
\end{equation*}
Note that $\V \P \V_{L^2}=|\b|+|\U|$ and $\V \P^2 \V_{L^2}\leq \frac{5}{4}(\U^2+\b^2)$. Then,
\begin{equation*}
\begin{split}
\V (a_0(\psi_1)&-a_0(\psi_2))\P+(b_0(\psi_1)-b_0(\psi_2))\P^2 \V_{L^2}\\
& \leq \left( \frac{1}{2}(1+|\b|)(\U^2+\b^2)(|\b|+|\U|) + \frac{2}{25}|\b|\frac{5}{4}(\U^2+\b^2)\right) c_{\U,\b}\V \psi_1-\psi_2\V_{\dot{H}^2} \\
&< (1+\U^2+\b^2)^2c_{\U,\b}\V \psi_1-\psi_2\V_{\dot{H}^2}.
\end{split}
\end{equation*}
Now, take $c_{\U,\b}:=\frac{1}{2}(1+\U^2+\b^2)^{-2}< \frac{1}{2}$ from which we deduce that
\begin{equation*}
\V K_\e(\psi_1)-K_\e(\psi_2)\V_{H^2}\leq \V \widetilde{K_\e}(\psi_1)-\widetilde{K_\e}(\psi_2)\V_{L^2}\leq \left( \frac{1}{2}+O(\e)\right)\V\psi_1-\psi_2\V_{H^2},
\end{equation*}
hence obtaining a contraction for $\e>0$ sufficiently small.
\end{proof}
\subsubsection{Non-triviality of the solution}\label{Non-triviality}
Given $\e>0$ small enough, let $\psi_\e\in X$ be the fixed point of $K_\e$, well-defined thanks to Proposition \ref{Prop Contraction}. We finish the proof of Theorem \ref{thm existence} showing that $\psi_\e$ is non-zonal and does not belong to $\textbf{Y}_2$. 
\begin{lemma}
Let $\e>0$ small enough so that, by Proposition \ref{Prop Contraction}, $\psi_\e$ is the fixed point of $K_\e$. Then,
\begin{equation*}
\begin{split}
\langle \psi_\e,Y_6^2\rangle&=-\b^2\frac{1}{36}\frac{45}{11\pi\sqrt{182}}c_{\U,\b}\e+O(\e^2), \\
\langle \psi_\e,Y_4^2\rangle &=-\frac{1}{14}\left( \frac{3\sqrt{3}}{154\pi}(11\U^2-5\b^2) -\b^2\frac{1}{7\pi}\frac{30\U^2}{7\U^2+15\b^2} \right)c_{\U,\b}\e + O(\e^2).
\end{split}
\end{equation*}
\end{lemma}
\begin{proof}
Recall that $\psi_\e$ solves
\begin{equation*}
\begin{split}
\D\psi_\e+6\psi_\e = &A(\psi_\e;\e)\P + B(\psi_\e;\e)\P^2 + c_{\U,\b}\P^3 \\
&+\e(\psi_\e+Y_2^2)(A(\psi_\e;\e)+2B(\psi_\e;\e)\P + 3c_{\U,\b}\P^2) \\
&+R(B(\psi_\e;\e),c_{\U,\b},\psi,\e).
\end{split}
\end{equation*}
We expand the right hand side of the above equation in orders of $\e$ and then in spherical harmonics, so that we can easily use
\begin{equation}\label{relation}
\langle (\D+6)\psi_\e,Y_n^m\rangle = (-n(n+1)+6)\langle \psi_\e,Y_n^m\rangle,
\end{equation}
for all $n\neq 2$. This yields
\begin{equation}\label{eq f eps}
\D\psi_\e+6\psi_\e = f_0+\e f_1 +O(\e^2)
\end{equation}
where
\begin{equation*}
\begin{split}
f_0&=\b^3\frac{45}{154\pi}\sqrt{\frac{5}{13}}c_{\U,\b} Y_6^0 +\b^2\U\frac{ 15\sqrt{33}}{154\pi}c_{\U,\b} Y_5^0 +\b\left( \frac{9(11\U^2+5\b^2)}{77\pi\sqrt{5}}c_{\U,\b} + \b\frac{3}{7\sqrt{\pi}}b_0\right)Y_4^0 \\
&+\U\left( \frac{3}{10\pi}\sqrt{\frac{3}{7}}(\U^2+5\b^2)c_{\U,\b} +3\b\sqrt{\frac{3}{35\pi}}b_0 \right)Y_3^0 +\left( \b a_0+\frac{7\U^2+5\b^2}{7\sqrt{5\pi}}b_0+3\b\frac{11\U^2+5\b^2}{28\pi}c_{\U,\b} \right)Y_2^0 \\
& + \U\left( \frac{3}{140\pi}(21\U^2+55\b^2)c_{\U,\b} +\b\frac{2}{\sqrt{5\pi}}b_0+a_0\right)Y_1^0 + \left(\b\frac{21\U^2+5\b^2}{14\pi\sqrt{5}}c_{\U,\b}+\frac{\U^2+\b^2}{2\sqrt{\pi}}b_0\right) Y_0^0 \\
\end{split}
\end{equation*}
and 
\begin{equation*}
\begin{split}
f_1&= \psi_\e|_{\e=0}(a_0+2b_0\P+3c_{\U,\b}\P^2) + a_1\P+b_1\P^2\\ 
&+ \b^2\frac{45}{11\pi\sqrt{182}}c_{\U,\b} Y_6^2 + \b\U\frac{3}{2\pi}\sqrt{\frac{15}{77}}c_{\U,\b} Y_5^2 + \left( \frac{3\sqrt{3}}{154\pi}(11\U^2-5\b^2)c_{\U,\b} +\b\frac{1}{7}\sqrt{\frac{15}{\pi}}b_0 \right) Y_4^2 \\
&+ \b\U\sqrt{\frac{3}{7\pi}}b_0Y_3^2 +\left(a_0-\b\frac{2}{7}\sqrt{\frac{5}{\pi}}b_0+3\frac{3\U^2+5\b^2}{28\pi}c_{\U,\b}\right)Y_2^2, \\
\end{split}
\end{equation*}
while the big $O(\e^2)$ notation is with respect to the $L^2(\Sph)$ norm.
We note that $f_0$ is automatically orthogonal to $Y_2^2$, while the orthogonality of $f_0$ with respect to $Y_2^0$ is satisfied if $a_0$ and $b_0$ are such that
\begin{equation*}
\b a_0+\frac{7\U^2+5\b^2}{7\sqrt{5\pi}}b_0+3\b\frac{11\U^2+5\b^2}{28\pi}c_{\U,\b}=0
\end{equation*}
so we can invert $(\D+6)$, find $\psi_\e|_{\e=0}=(\D+6)^{-1}f_0$ and observe that, in particular, it is zonal. This in turn implies that $f_1$ is orthogonal to $Y_2^2$ as long as
\begin{equation*}
a_0-\b\frac{2}{7}\sqrt{\frac{5}{\pi}}b_0+3\frac{3\U^2+5\b^2}{28\pi}c_{\U,\b}=0.
\end{equation*}
The system is thus solved for 
\begin{equation*}
\begin{split}
a_0&=-\frac{1}{28\pi}\left( 9\U^2+15\b^2 + \frac{240\b^2\U^2}{7\U^2+15\b^2} \right)c_{\U,\b}, \\
b_0&=-\b c_{\U,\b}\frac{6\U^2}{7\U^2+15\b^2}\sqrt{\frac{5}{\pi}}.
\end{split}
\end{equation*}
This gives precise values for $f_0$ above and thus for
\begin{equation*}
\langle \psi_\e|_{\e=0},Y_n^m\rangle = \frac{1}{6-n^2-n}\langle f_0, Y_n^m \rangle, \quad n\neq 2.
\end{equation*}
One may keep proceeding in this fashion and obtain equations for the coefficients $a_j$ and $b_j$ by requiring that $\langle f_j, Y_2^0\rangle =0$ and $\langle f_{j+1}, Y_2^2\rangle =0$, for all $j\geq 1$, formally constructing the coefficients $A(\psi;\e)$ and $B(\psi;\e)$ obtained in Lemma \ref{lemma def A B}.
Furthermore, testing the right hand side of \eqref{eq f eps} and using the relation \eqref{relation} we find that, among others, 
\begin{equation*}
\begin{split}
\langle \psi_\e,Y_6^2\rangle&=-\b^2\frac{1}{36}\frac{45}{11\pi\sqrt{182}}c_{\U,\b}\e+O(\e^2), \\
\langle \psi_\e,Y_4^2\rangle &=-\frac{1}{14}\left( \frac{3\sqrt{3}}{154\pi}(11\U^2-5\b^2) -\b^2\frac{1}{7\pi}\frac{30\U^2}{7\U^2+15\b^2} \right)c_{\U,\b}\e + O(\e^2).
\end{split}
\end{equation*}
We remark that $\psi_\e$ is non-zonal and non-trivial, for all $\b,\U\in\R$ such that $\b^2+\U^2>0$. Indeed, for any $\b\neq0$ we have that $\langle \psi_\e,Y_6^2\rangle$ never vanishes and for $\b=0$ it must be $\U\neq 0$, hence $\langle \psi_\e,Y_4^2\rangle$ is non-zero. 
\end{proof}
\begin{remark}
The procedure carried out for the proof of Theorem \ref{thm existence} may be adapted to prove similar results for $\Psi_n$ as follows. As before, the problem reduces to finding non-zonal solutions $\psi$ to $\D\psi + \lambda_n\psi= f(\Psi_n+\e\psi)$. Introducing some (possibly a combination of) $Y_n^m$ through the elliptic operator, we require the nonlinearity $f(\Psi_n+\e\psi+\e Y_n^m)$ to be orthogonal to the space of spherical harmonics $\textbf{Y}_n$, of eigenvalue $\lambda_n$. This may be achieved by considering $f$ to be a polynomial of high enough degree, whose coefficients are to be determined so that both all orthogonality conditions and the contraction mapping property are satisfied. The case $\Psi_*=\Psi_3$ should be the easiest to tackle, since for that stream-function one may introduce $Y_3^2$ through $\D+12$ and choose the nonlinearity $f$ and the orthogonal assumptions for $\psi$ to be the same ones as for the case $\Psi_*=\Psi_2$.
\end{remark}
\subsection{Analytic Regularity}\label{sec reg}
To show Theorem \ref{thm reg}, we will demonstrate that the solution $\psi_\e$ constructed in Theorem \ref{thm existence} is a real analytic function whose $\G_\lambda$ norm is uniformly bounded in $\e$ for some $\lambda>0$. This is the statement of Proposition \ref{prop gevrey reg sol} below. Since the stream function $\Psi_\e=\Psi_*+\e\psi_\e$, this will directly yield that
\begin{equation*}
\V \b Y_2^0+\U Y_1^0 -\Psi_\e\V_{\G_\lambda}\leq M\e.
\end{equation*}
Before Proposition \ref{prop gevrey reg sol}, we recall the operator $A=-\D+1$ and we first show a more generic result concerning analytic regularity and $\G_\lambda$ bounds for solutions to semilinear elliptic equations with analytic coefficients on $\Sph$.
\begin{lemma}\label{lemma analytic regularity}
For $0\leq n\leq N$ let $a_n\in C^\o(\Sph)$ be analytic functions such that 
\begin{equation}\label{gevr norm a}
\V a_n \V_{\G_\lambda}\leq C_1e^{C_2\lambda}, \quad 0\leq n\leq N,
\end{equation}
for some $C_1,C_2 >0$ and for all $\lambda\geq0$. Let $u\in H^2$ solve the semilinear partial differential equation
\begin{equation}\label{eq u}
Au=\sum_{n=0}^N a_nu^n.
\end{equation}
Then, $u\in C^\o(\Sph)$ and there exists $\lambda_*>0$ depending only on $\V u \V_{H^2},\,N,\,C_1,\,C_2$ such that
\begin{equation*}
\V u \V_{\G_\lambda}\leq 4\V u \V_{H^2},
\end{equation*}
for all $0\leq \lambda\leq \lambda_*$.
\end{lemma}
\begin{proof}
We begin with the analyticity of $u$, for which we first prove that $u\in C^\infty$, it is smooth. Let $k\geq 2$ and let $C_k$ be the algebra constant of $H^k$. Then,
\begin{equation*}
\V u\V_{H^{k+2}}=\V Au \V_{H^k} \leq \sum_{n=0}^N \V a_n u^n\V_{H^k}\leq N\left(\max_{n=0,...,N}\V a_n\V_{H^k}\right)\left(1+(C_k\V u\V_{H^k})^N\right).
\end{equation*}
Since $u\in H^2$, the usual bootstrap argument derived from the above inequality yields that $u\in H^k$, for all $k\geq 0$, from which we deduce that $u\in C^\infty$, it is smooth.
 
Now, for any point $p\in\Sph$ distinct from the north and south poles, we find that $u$ is smooth and solves
\begin{equation*}
-\left(\partial_{\theta\theta}+\frac{\cos\theta}{\sin\theta}\partial_\theta +\frac{1}{\sin^2\theta}\partial_{\varphi\varphi}\right)u(\theta,\varphi)+u(\theta,\varphi)=\sum_{n=0}^N a_n(\theta,\varphi) u^n(\theta,\varphi).
\end{equation*}
in local coordinates $(\theta,\phi)$ in a small enough neighbourhood of $p\in \Sph$ such that it does not intersect the north and south poles. Together with \eqref{gevr norm a}, the above equation shows that $u$ is a smooth solution to a semilinear elliptic partial differential equation with analytic coefficients in an open neighbourhood of $\mathbb{R}^2$, for which classic results yield that the solution is in fact real analytic in that neighbourhood (see \cite{Blatt, Hashimoto} for modern proofs). For the poles, we use the coordinate system \eqref{equator chart} to remove the artificial singularities appearing in the expression of the Laplacian of $u$ in local coordinates. Since $\Sph$ is compact, a classical partition of unity and covering argument shows that $u$ is real analytic in the whole of $\Sph$ and thus $\Vert u \Vert_{\G_\lambda}<\infty$, for all $0\leq \lambda<\lambda_1$, for some $\lambda_1>0$. We now find $\e$-independent estimates on $\V u \V_{\G_\lambda}$. In particular, from the Gevrey norm definition \eqref{def gevrey norm}, we have that
\begin{equation*}
\frac{1}{2}\frac{d}{d\lambda}\V u \V_{\G_\lambda}^2=\left\langle A^{1/2}Ae^{\lambda A^{1/2}}u,Ae^{\lambda A^{1/2}}u\right\rangle\leq \V A^{1/2}Ae^{\lambda A^{1/2}}u \V_{L^2} \V u \V_{\G_\lambda},
\end{equation*}
from which we deduce that
\begin{equation*}
\frac{d}{d\lambda}\V u \V_{\G_\lambda}\leq \V A^{1/2}Ae^{\lambda A^{1/2}}u \V_{L^2} \leq \V Ae^{\lambda A^{1/2}}Au \V_{L^2}=\left\V \sum_{n=0}^N a_nu^n \right\V_{\G_\lambda} \leq \sum_{n=0}^N \V a_nu^n \V_{\G_\lambda}.
\end{equation*}
Now, let $C>1$ be the algebra constant of $\G_\lambda$. Then, for all $0\leq n\leq N$ we have that
\begin{equation*}
\V a_nu^n\V_{\G_\lambda}\leq \V a_n \V_{\G_\lambda}(C\V u \V_{\G_\lambda})^n \leq C^NC_1e^{C_2\lambda}(1+\V u \V_{\G_\lambda}^N),
\end{equation*}
where we have used the uniform bounds in \eqref{gevr norm a}. Therefore,
\begin{equation*}
\frac{d}{d\lambda}\V u \V_{\G_\lambda}\leq NC^NC_1e^{C_2\lambda}(1+\V u \V_{\G_\lambda}^N), 
\end{equation*}
multiplying both sides by $\V u \V_{\G_\lambda}^{N-1}$ and using $\V u \V_{\G_\lambda}^{N-1}\leq 1 + \V u \V_{\G_\lambda}^N$ gives
\begin{equation*}
\frac{d}{d\lambda}(1+\V u \V_{\G_\lambda}^N)\leq N^2C^NC_1e^{C_2\lambda}(1+\V u \V_{\G_\lambda}^N)^2.
\end{equation*}
Solving the differential inequality and recalling that $\V u \V_{\G_0}=\V u \V_{H^2}$, we obtain
\begin{equation*}
\V u \V_{\G_\lambda}\leq \left( \frac{\V u \V_{H^2}^N+(1+\V u \V_{H^2}^N)N^2C^N\frac{C_1}{C_2}\left(e^{C_2\lambda}-1\right)}{1-(1+\V u \V_{H^2}^N)N^2C^N\frac{C_1}{C_2}\left(e^{C_2\lambda}-1\right)} \right)^{\frac{1}{N}}.
\end{equation*}
Finally, let $\lambda_*>0$ be such that $(1+\V u \V_{H^2}^N)N^2C^N\frac{C_1}{C_2}\left(e^{C_2\lambda_*}-1\right)= \min\left\lbrace\frac{1}{2}, \V u \V_{H^2}^N\right\rbrace$, so that
\begin{equation*}
\V u \V_{\G_\lambda}\leq 4\V u \V_{H^2},
\end{equation*}
for all $0\leq \lambda \leq \lambda_*$.
\end{proof}
With the above Lemma at hand, we are now able to prove the following result and finish the section.
\begin{proposition}\label{prop gevrey reg sol}
Let $\e>0$ small enough and let $\psi_\e\in H^2$ be the fixed point of $K_\e$. Then, $\psi_\e\in C^\o(\Sph)$ and there exists $\lambda>0$ independent of $\e>0$ such that 
\begin{equation*}
\V \psi_\e \V_{\G_\lambda}\leq 400(1+\U^2+\b^2)^2.
\end{equation*} 
\end{proposition}
\begin{proof}
Fix $\e>0$ as in the statement of Theorem \ref{thm existence}. Then, $\psi_\e$ solves
\begin{equation*}
\D\psi_\e+6\psi_\e=f(\b Y_2^0+\U Y_1^0+\e\psi_\e+\e Y_2^2),
\end{equation*}
where $f(s)=As +Bs^2+c_{\U,\b}s^3$, with $A=A(\psi_\e;\e)$ and $B=B(\psi_\e;\e)$ being fixed coefficients, uniformly bounded in $\e>0$ for $\e$ small enough, by virtue of Lemma \ref{lemma prop A B}. Writing $f$ as a polynomial in $\psi_\e$ one obtains uniformly bounded $c_{i,j}\in \R$ for $0\leq i,j\leq 3$ such that
\begin{equation*}
\D\psi_\e+6\psi_\e=\sum_{n=0}^3\left(\sum_{i+j=3-n}c_{i,j}(\b Y_2^0 + \U Y_1^0)^i\e^j(Y_2^2)^j\right)\e^n\psi_\e^n,
\end{equation*}
from which we deduce that
\begin{equation*}
A\psi_\e=7\psi_\e-\sum_{n=0}^3\left(\sum_{i+j=3-n}c_{i,j}(\b Y_2^0 + \U Y_1^0)^i\e^j(Y_2^2)^j\right)\e^n\psi_\e^n=:\sum_{n=0}^3a_nu^n.
\end{equation*}
Now, each $a_n$ is a finite combination of spherical harmonics, which renders the analyticity of the coefficients. Moreover, for all $m\geq 0$ and $0\leq l\leq m$, one has
\begin{equation*}
\V Y_m^l \V_{\G_\lambda}=\V Ae^{\lambda A^{1/2}}Y_m^l\V_{L^2}=(m^2+m+1)e^{\lambda(m^2+m+1)^{1/2}},
\end{equation*}
which yields 
\begin{equation*}
\V a_n \V_{\G_\lambda}\leq C_1e^{C_2\lambda}, \quad 0\leq n\leq 3,
\end{equation*}
for some $C_1>0$ and $C_2>0$ large enough. Consequently, we can apply Lemma \ref{lemma analytic regularity} to obtain the analyticity of $\psi_\e$ and the Gevrey bound $\V \psi_\e \V_{\G_\lambda}\leq 4\V \psi_\e \V_{H^2}$, for all $0\leq \lambda \leq \lambda_*$. Since $\psi_\e\in X$, we conclude that $\V \psi_\e \V_{\G_\lambda}\leq 400(1+\U^2+\b^2)^2$, for all $ 0\leq \lambda \leq \lambda_*$. The proof is finished.
\end{proof}
\section{Stationary Structures near rigid rotation}\label{sec rigidity}
This section is devoted to the proof of Theorem \ref{main thm 2 rigidity Y_1^0}, which concerns the rigidity of the base zonal flow $\a Y_1^0$. Afterwards, we discuss the effectiveness of condition \eqref{spec cond} to geometrically describe the set of solutions near $\a Y_1^0$. 

\subsection{Rigidity}
The idea of the proof is to obtain a coercive estimate for the linearised operator related to the Euler equation on the rotating sphere about $\alpha Y_1^0$ and to simultaneously control the non-linear term, obtaining a contradiction if the associated vorticity is both non-zonal and sufficiently close to the rigid rotation vorticity $-2\alpha Y_1^0$.

\begin{proof}[Proof of Theorem \ref{main thm 2 rigidity Y_1^0}]
A general longitudinal travelling wave solution $U(\theta,\varphi,t)$ to the 2D Euler equation on the sphere is of the form
\begin{equation*}
U(\theta,\varphi,t)=U_\theta(\theta,\varphi-ct)\textbf{e}_\theta + U_\varphi(\theta,\varphi-ct)\textbf{e}_\varphi,
\end{equation*}
for some $c\in\R$, and its associated vorticity is given by
\begin{equation*}
\O(\theta,\varphi-ct)=-\frac{1}{\sin\theta}\partial_\theta(\sin\theta U_\varphi(\theta,\varphi-ct)) + \frac{1}{\sin\theta}\partial_\varphi U_\theta(\theta,\varphi-ct).
\end{equation*}
We begin by setting
\begin{equation}\label{def O and o}
\O = -2\alpha Y_1^0 + \widetilde{\o}, \quad \widetilde{\o}=\o+\widetilde{\o}_0, \quad \widetilde{\o}_0=\frac{1}{2\pi}\int_0^{2\pi}\widetilde{\o}\,\d\varphi.
\end{equation}
Note that
\begin{equation*}
\int_{\Sph} \O\, \d\Sph= \int_0^{2\pi}\int_0^\pi \left(-\frac{1}{\sin\theta}\partial_\theta(\sin\theta U_\varphi) + \frac{1}{\sin\theta}\partial_\varphi U_\theta\right)\sin\theta\, \d\theta\, \d\varphi=0,
\end{equation*}
and further observe $\int_0^{2\pi}\o\,\d\varphi=0$. Hence, since $\O$ is spherically average free we define $\Psi:=\D^{-1}\O$, with $\int_{\Sph} \Psi\, \d\Sph=0$. We also deduce that both $\widetilde \o$ and $\o$ are spherically average free, so that we can further define
\begin{equation*}
\Psi = \alpha Y_1^0 +\widetilde{\psi}, \quad \widetilde{\psi}=\psi+\widetilde{\psi}_0,\quad  \widetilde{\psi}_0=\frac{1}{2\pi}\int_0^{2\pi}\widetilde{\psi}\,\d\varphi,
\end{equation*} 
where $\widetilde{\psi} = \D^{-1}\widetilde{\o}$ and $\psi=\D^{-1}\o$ are such that $\int_{\Sph} \widetilde{\psi}\, \d\Sph=0$ and  $\int_0^{2\pi}\psi\, \d\varphi=0$. Additionally, we set $\widetilde{u}=u+\widetilde{u}_0$, where $u=\nabla^\perp\psi$ and $\widetilde{u}_0=\nabla^\perp\widetilde{\psi}_0$, respectively.

In particular, all functions we are considering are average free on the sphere, so that the $\V \cdot \V_{H^k(\Sph)}$ and $\V \cdot \V_{\dot{H}^k(\Sph)}$ norms are equivalent. Moreover, we note that the smallness assumption of Theorem \ref{main thm 2 rigidity Y_1^0} now reads
\begin{equation}\label{smallnes omega}
\Vert \widetilde{\o}\Vert_{H^4}\leq \e_0
\end{equation}
Writing the Euler equation on the rotating sphere in vorticity form we obtain
\begin{equation*}
\begin{split}
0&=\partial_t\O+ U\cdot\nabla (\O-4\U Y_1^0) \\
&=(\a a-c)\partial_\varphi \o + a(2\a+4\U)\partial_\varphi\psi + u\cdot\nabla \o +\frac{1}{\sin\theta}\partial_\varphi\psi\partial_\theta\widetilde{\o}_0-\partial_\theta\widetilde{\psi}_0\frac{1}{\sin\theta}\partial_\varphi\o.  \\
\end{split}
\end{equation*}
Let us further define the linear operator $\L \o :=(a\a-c)\partial_\varphi\o + a(2\a+4\U)\partial_\varphi\psi$. Inspecting $\L\o$ in spherical harmonics one sees that the choice of $c$ and $\U$ in \eqref{spec cond} ensures the existence of some constant $C_1=C_1(c,\U)>0$ such that
\begin{equation*}
\V \partial_\varphi \o \V_{L^2}\leq C_1 \V \L\o \V_{L^2}.
\end{equation*}
On the other hand, the Sobolev embedding yield
\begin{equation*}
\V u\cdot\nabla \o \V_{L^2} = \V \nabla^\perp\psi\cdot\nabla\o \V_{L^2}\leq \V\nabla\psi\V_{L^\infty} \V \nabla \o \V_{L^2}
\lesssim \V \psi \V_{\dot{H}^3}\V \o\V_{\dot{H}^1}
=\V \o \V_{\dot{H}^1}^2, 
\end{equation*}
while the interpolating inequality between Sobolev spaces and Poincar\'{e} inequality provides
\begin{equation*}
\V \o \V_{\dot{H}^1}^2 \lesssim \V \o \V_{L^2} \V \o \V_{\dot{H}^2} 
\lesssim \V \partial_\varphi\o \V_{L^2} \V \o \V_{\dot{H}^2}.
\end{equation*}
Next, we show that $\left\V \frac{1}{\sin\theta}\partial_\theta \widetilde{\o}_0 \right\V_{L^\infty}\lesssim \V \widetilde{\o}\V_{\dot{H}^4}$. Indeed, for $\theta\in \left(\frac{\pi}{4},\frac{3\pi}{4}\right)$ we have that $\left|\frac{1}{\sin\theta}\right|\lesssim 1$ and using the definition of $\widetilde{\o}_0$ and the Sobolev embedding we can easily estimate
\begin{equation*}
\left| \frac{1}{\sin\theta}\partial_\theta \widetilde{\o}_0 \right|\lesssim \V\partial_\theta \widetilde{\o}_0\V_{L^\infty} \leq \V\partial_\theta \widetilde{\o}\V_{L^\infty} \leq \V\nabla \widetilde{\o}\V_{L^\infty}\lesssim \V\widetilde{\o}\V_{\dot{H}^3}.
\end{equation*}
Similarly, for $\theta\in \left[ 0,\frac{\pi}{4}\right)\cup\left(\frac{3\pi}{4}, \pi\right]$ we have that $\left| \frac{1}{\cos\theta}\right|\lesssim 1$ and we bound
\begin{equation*}
\begin{split}
\left| \frac{1}{\sin\theta}\partial_\theta \widetilde{\o}_0 \right| \lesssim \left| \frac{1}{2\pi}\int_0^{2\pi}\frac{\cos\theta}{\sin\theta}\partial_\theta\widetilde{\o}\,\d\varphi\right|&=\left| \frac{1}{2\pi}\int_0^{2\pi}(\D\widetilde{\o} - \partial^2_\theta\widetilde{\o}) \,\d\varphi\right| \\
&\leq \V \D\widetilde{\o}\V_{L^\infty} + \V \widetilde{\o} \V_{C^2} \\
&\lesssim \V \widetilde{\o}\V_{\dot{H}^4}
\end{split}
\end{equation*}
Therefore, we easily estimate
\begin{equation*}
\left\V \frac{1}{\sin\theta}\partial_\varphi\psi\partial_\theta\widetilde{\o}_0\right\V_{L^2}\leq \left\V \frac{1}{\sin\theta}\partial_\theta\widetilde{\o}_0\right\V_{L^\infty}\V \partial_\varphi\psi\V_{L^2}\lesssim \V \widetilde{\o}\V_{\dot{H}^4}\V \partial_\varphi \psi \V_{L^2},
\end{equation*}
and in the same manner we also bound
\begin{equation*}
\left\V \frac{1}{\sin\theta}\partial_\varphi\o\partial_\theta\widetilde{\psi}_0\right\V_{L^2}\lesssim \V \widetilde{\psi}\V_{\dot{H}^4}\V \partial_\varphi \o \V_{L^2} = \V \widetilde{\o}\V_{\dot{H}^2}\V \partial_\varphi \o \V_{L^2}.
\end{equation*}
Now, we have that
\begin{equation*}
\begin{split}
\V \partial_\varphi\o \V_{L^2}\lesssim \V \L\o \V_{L^2}&\leq \V u\cdot \nabla\o \V_{L^2}+ \left\V \frac{1}{\sin\theta}\partial_\varphi\psi\partial_\theta\widetilde{\o}_0\right\V_{L^2}+\left\V \frac{1}{\sin\theta}\partial_\varphi\o\partial_\theta\widetilde{\psi}_0\right\V_{L^2} \\
&\lesssim \V \partial_\varphi\o \V_{L^2} \V \o \V_{\dot{H}^2} + \V \widetilde{\o}\V_{\dot{H}^4}\V \partial_\varphi \psi \V_{L^2} + \V \widetilde{\o}\V_{\dot{H}^2}\V \partial_\varphi \o \V_{L^2}.
\end{split}
\end{equation*}
We further observe that $\V\o\V_{\dot{H}^2}\leq 2 \V \widetilde \o \V_{\dot{H}^2}$ and $\V \partial_\varphi\psi\V_{L^2}\leq \V \partial_\varphi	\o \V_{L^2}$. Therefore, there exists a constant $C>1$ such that
\begin{equation*}
\V \partial_\varphi \o \V_{L^2}\leq C\V \partial_\varphi \o \V_{L^2}\V \widetilde{\o}\V_{\dot{H}^4}.
\end{equation*} 
Choosing $\e_0=\frac{1}{2C}$ shows that the above inequality is satisfied only if $\partial_\varphi\o=0$. Hence, $\O$ is a zonal function and the associated velocity field $U$ is a zonal flow, namely $U=U(\theta)\textbf{e}_\varphi$.
\end{proof}
We finish this subsection by proving Corollary \ref{main corollary 3}. 
\begin{proof}[Proof of Corollary \ref{main corollary 3}]
The result follows swiftly from the proof of Theorem \ref{main thm 2 rigidity Y_1^0} above. In the setting for which there is no rotation and the solution is assumed to be steady (which corresponds to $\U=0$ and $c=0$), condition \eqref{spec cond} fails for $n=1$. Still, in this case the linear operator $\L\o:=a\a(1+2\D^{-1})\partial_\varphi\o$ admits the coercive estimate
\begin{equation*}
\Vert \partial_\varphi\o \Vert_{L^2}\leq\frac{2}{3}a\a\Vert \L\o\Vert_{L^2}
\end{equation*}
as long as $\o$ (equivalently $\O$ as defined in \eqref{def O and o}) is orthogonal to $Y_1^m$, for $|m|=1$. With this estimate at hand, one can replicate the proof of Theorem \ref{main thm 2 rigidity Y_1^0} to obtain the desired result.
\end{proof}
\subsection{On the sharpness of Theorem \ref{main thm 2 rigidity Y_1^0}}
There exist pairs $(c,\a)$ for which relation \eqref{spec cond} does not hold. For example, $c=0$ and $\a=\U$ is one of them. For this choice, the linearised operator acting on the vorticity arising from the Euler equations on the rotating sphere around the flow given by the stream-function $\a Y_1^0=\U Y_1^0$ has a non-trivial kernel formed by zonal flows and eigenfunctions of the Laplace-Beltrami operator of eigenvalue $-6$. The steady non-trivial and non-zonal solutions to the Euler equations \eqref{Euler Sphere U} arbitrarily close to $\U Y_1^0$ obtained in Corollary \ref{main corollary 1} are an easy consequence of Theorem \ref{main thm 1} for $\b=0$ and provide an example for which the spectral condition \eqref{spec cond} fails and the conclusion of Theorem \ref{main thm 2 rigidity Y_1^0} does not hold.

The pair $c=\eta$ and $\a=3\sqrt{\frac{\pi}{3}}\eta$, for any $\eta\in\R\setminus\lbrace 0\rbrace$ gives a choice of values for which \eqref{spec cond} fails for the Euler equations on a non-rotating sphere, that is, $\U=0$. The non-trivial travelling wave solutions obtained in Corollary \ref{main corollary 2} are a further example in which the conclusion of Theorem \ref{main thm 2 rigidity Y_1^0} is not true.

Another setting for which we have equality in \eqref{spec cond} is when $c=0$ and $\U=0$. In this case the linearised operator around $Y_1^0$ (we can take $\a=1$ for simplicity) has a non-trivial kernel formed by zonal flows and eigenfunctions of the laplacian of eigenvalue $-2$. However, as we have seen from Corollary \ref{main corollary 3} above, the conclusion of Theorem \eqref{main thm 2 rigidity Y_1^0} is still valid if one assumes them to be orthogonal to $Y_1^m$, for $|m|=1$. Alternatively, dropping this assumption, one would be tempted to follow the strategy presented for the proof of Theorem \ref{main thm 1} in Section \ref{sec main thm 1} and consider, for instance, 
\begin{equation*}
\Psi_\e= Y_1^0+\e\psi,\quad F_\e(s) = -2s + \e f(s),
\end{equation*}
which yields the non-linear elliptic equation
\begin{equation*}
\D\psi + 2\psi = f(Y_1^0+\e\psi),
\end{equation*}
and try to construct non-zonal non-trivial solutions by introducing $Y_1^1$ through the kernel, thus obtaining
\begin{equation*}
\D\psi + 2\psi = f( Y_1^0+\e\psi +\e Y_1^1).
\end{equation*}
As before, one would need to construct both $f$ and $\psi$ simultaneously. However, some problems arise. Firstly, one cannot expect $f$ to remain automatically orthogonal to $Y_1^1$ any more, and hence an additional compatibility condition would be in place. 

In that case, looking for some general $f=f(Y_1^0+\e\psi+\e Y_1^1)$ with $f\in C^1$ and expanding in orders of $\e$, one would have
\begin{equation*}
f=f(Y_1^0)+\e(\psi+Y_1^1)f'(Y_1^0)+O(\e^2)
\end{equation*}
and in this case the compatibility conditions would read 
\begin{equation*}
\langle f,Y_1^0\rangle = \langle f,Y_1^1\rangle =0.
\end{equation*}
These two compatibility conditions are related to each other
due to an intrinsic property of the spherical harmonics that can be traced back directly to the Legendre polynomials. Indeed, one can prove that for all $n\geq 1$ there exists $C>0$ only depending on $n$ such that
\begin{equation*}
\langle f(Y_n^0),Y_n^0 \rangle= C \langle f'(Y_n^0)Y_n^1,Y_n^1 \rangle.
\end{equation*}
Indeed, note that from \eqref{eig eq Leg pol} and the definitions of the spherical harmonics and the associated Legendre polynomials,
\begin{equation*}
\begin{split}
\langle f(Y_n^0),Y_n^0 \rangle & = \pi\int_0^{2\pi}f(Y_n^0)Y_n^0(\theta)\sin\theta \,\d\theta \\
&= \pi C \int_0^{\pi}f(Y_n^0)P_n(\cos\theta)\sin\theta \,\d\theta \\
&= \pi C \int_0^{\pi}f(Y_n^0)\frac{\d}{\d\cos\theta}\left( \sin^2\theta \frac{\d}{\d\cos\theta} P_n(\cos\theta)\right)\sin\theta \,\d\theta \\
&=\pi C \int_0^{\pi}f(Y_n^0)\frac{\d}{\d\theta}\left( \sin^2\theta \frac{\d}{\d\cos\theta} P_n(\cos\theta)\right) \,\d\theta.
\end{split}
\end{equation*}
Integrating by parts and recognising $Y_n^1$, we obtain
\begin{equation*}
\begin{split}
\langle f(Y_n^0),Y_n^0 \rangle &=-\pi C \int_0^{\pi}f'(Y_n^0)\frac{\d}{\d\theta}Y_n^0(\theta)\left( \sin^2\theta \frac{\d}{\d\cos\theta} P_n(\cos\theta)\right) \,\d\theta \\
&= \pi C \int_0^{\pi}f'(Y_n^0)\frac{\d}{\d\cos\theta}P_n(\cos\theta)\left( \sin^2\theta \frac{\d}{\d\cos\theta} P_n(\cos\theta)\right) \sin\theta\,\d\theta \\
&= C\int_0^\pi \int_0^{2\pi}f'(Y_n^0)(P_n^1(\cos\theta))^2\cos^2\varphi \sin\theta \,\d\theta \,\d\varphi \\
&= C\langle f'(Y_n^0)Y_n^1,Y_n^1\rangle.
\end{split} 
\end{equation*}
This has a non-trivial consequence, since then the compatibility conditions read
\begin{equation*}
\begin{split}
0=\langle f, Y_1^0\rangle &= \langle f(Y_1^0),Y_1^0 \rangle + O(\e), \\
0=\langle f, Y_1^1\rangle &= \langle f'(Y_1^0)Y_1^1,Y_1^1\rangle + \langle f'(Y_1^0)\psi,Y_1^1\rangle + O(\e) \\
&= C\langle f(Y_1^0),Y_1^0 \rangle + \langle f'(Y_1^0)\psi,Y_1^1\rangle + O(\e).
\end{split}
\end{equation*}
The two conditions can therefore be written as
\begin{equation*}
0=\langle f(Y_1^0),Y_1^0 \rangle + O(\e), \quad 0=\langle f'(Y_1^0)\psi,Y_1^1\rangle + O(\e).
\end{equation*}
Any choice of polynomial $f$ or, more generally, any linear combination of functions, shows that we cannot start a completely determined recursive definition of the coefficients because we lack information on $\langle f'(Y_1^0)\psi,Y_1^1\rangle$. For this reason, we cannot apply our strategy to construct stationary non-trivial non-zonal solutions to the non-rotating Euler equation arbitrarily close to $Y_1^0$. 

Eventually, our method is doomed to fail. Indeed, assuming that $\D\Psi_\e=F_\e(\Psi_\e)$ is a stationary non-trivial non-zonal solution with $\Psi\in H^2(\Sph)$, our analytic regularity results would imply that $f(\Psi_\e)$ is uniformly bounded in $\e$ and thus the non-linearity $F_\e$ would be such that $F_\e'>-6$. However, using \cite[Theorem 4]{Constantin-Germain}, any solution $\D\Psi=G(\Psi)$ with $G'>-6$ must be zonal, which yields a contradiction. 

The key difference with Theorem \ref{main thm 1} and Corollary \ref{main corollary 1} above is that in these cases we have more elements in the kernel, namely we can introduce $Y_2^2$ via $\D+6$, for this choice we obtain two completely determined compatibility conditions and the non-linearity $F_\e$ is no longer such that $F_\e'>-6$.
\subsubsection*{Acknowledgements}
The author thanks Michele Coti Zelati and Martin Taylor for helpful conversations and insightful comments on the development of this paper.

\bibliographystyle{abbrv}
\bibliography{refs}
\end{document}